\documentclass[11pt]{article}

\newcommand{\C}{\mathcal{C}_{X,N}}
\newcommand{\vsp}{\vspace{\headheight}}
\newcommand{\Rinf}{\mathbb{R}^{\infty}}
\newcommand{\R}{\mathbb{R}}
\newcommand{\tCb}{\tilde{\mathcal{C}}^{b}_{X,N}}

\usepackage{mathrsfs, amsfonts, amsmath, amssymb, amsthm}
\usepackage[cmtip,arrow]{xy}
\usepackage{pb-diagram,pb-xy}
\usepackage{pstricks}
\usepackage{pst-all}
\usepackage{graphicx}
\usepackage{bbm}

\theoremstyle{definition} \newtheorem{definition}{Definition}
\theoremstyle{plain} \newtheorem{lemma}[definition]{Lemma}
\newtheorem{theorem}[definition]{Theorem}
\newtheorem{corollary}[definition]{Corollary}
\theoremstyle{definition}
 \theoremstyle{plain}
\newtheorem{proposition}[definition]{Proposition}
\theoremstyle{remark} 
\theoremstyle{plain}

\begin{document}

\title{An open-closed cobordism category with background space}

\author{Elizabeth Hanbury \thanks{Department of Mathematics, University of Durham, South Road, Durham, DH1 3LE, UK. Email: liz.hanbury@gmail.com.} \thanks{Partially supported by EPSRC Doctoral Training Grant EP/P500397/1 and National University of Singapore Research Grant R-146-000-097-112} }

\maketitle

\abstract{In this paper we introduce an open-closed cobordism category with maps to a background space. We identify the classifying space of this category for certain classes of background space. The key ingredient is the homology stability of mapping class groups with twisted coefficients.}

\tableofcontents

\section{Introduction}

In the late seventies and early eighties, cobordism categories were used by Atiyah and Segal to give an axiomatic description of field theories. Indeed, much of the motivation to study cobordism categories, and many of the indications of how we should do so, come from mathematical physics. Informally, the relationship between mathematical physics and cobordism theory can be seen in the following; if we track the motion of a collection of strings through time then the space they sweep out, their worldsheet, is precisely a cobordism.

Cobordism categories have been studied in algebraic topology in connection with mapping class groups and infinite loop space theory. The first work of this kind was Tillmann's result in \cite{Tillmann} relating the classifying space of the $2$-dimensional cobordism category to the classifying space of the stable mapping class group. Then in \cite{4author} the authors identified the classifying space of the $d$-dimensional cobordism category for any $d \geq 1$, showing that it is homotopy equivalent to a familiar infinite loop space. Together, these two results give a new proof of Madsen and Weiss' generalization of the Mumford conjecture which allows the calculation of the cohomology of the stable mapping class group.

The cobordism categories studied in those papers feature cobordisms between \emph{closed} manifolds. These pertain to closed field theories in physics. Also of interest are \emph{open-closed} field theories and, correspondingly, open-closed cobordisms i.e. cobordisms between manifolds that may have boundary. Just as the worldsheet of a collection of (closed) strings is an (ordinary) cobordism, so the worldsheet of a collection of open and closed strings is an open-closed cobordism.

The definition is as follows. An oriented open-closed cobordism between two oriented $(d-1)$-manifolds $M_{1}$ and $M_{2}$ is an oriented $d$-dimensional manifold $W$ such that
\begin{enumerate}
\item[\textit{(i)}] $\partial W = (\partial_{in}W \sqcup \partial_{out}W) \cup \partial_{f}W$
\item[\textit{(ii)}] There are diffeomorphisms $\partial_{in}W \cong M_{1}$ and $\partial_{out}W \cong -M_{2}$ where $-M_{2}$ means $M_{2}$ with the opposite orientation.
\item[\textit{(iii)}] $(\partial_{in}W \sqcup \partial_{out}W) \cap \partial_{f}W = \partial (\partial_{f}W) = \partial(\partial_{in}W \sqcup \partial_{out}W).$
\end{enumerate}
An example of an open-closed cobordism is given in Figure \ref{exoccob}.
\begin{figure}
\begin{center}
\includegraphics{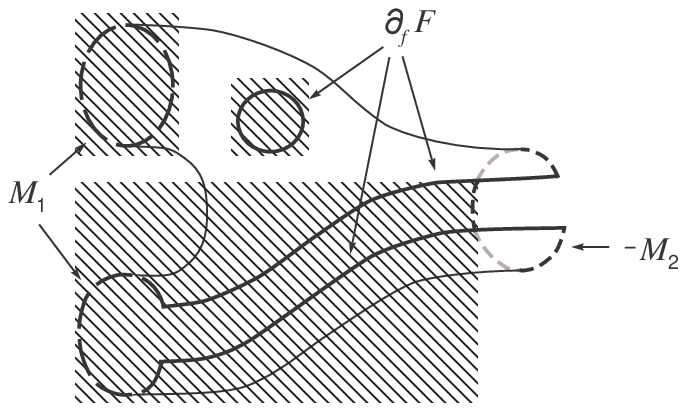}
\end{center}
\label{exoccob}
\caption{An open-closed cobordism}
\end{figure}
We call $\partial_{f}W$ the free boundary of $W$ and the closed components of $\partial_{f}W$ \emph{windows}. The open-closed cobordism in Figure \ref{exoccob} has one window. We think of the boundary of an open-closed cobordism from $M_{1}$ to $M_{2}$ as the union of three pieces; $M_{1}$, $M_{2}$ and the free boundary. Note that in particular $(ii)$ and $(iii)$ imply that $\partial_{f}W$ is a cobordism from $\partial M_{1}$ to $\partial M_{2}$.

Open-closed cobordism categories have recently been studied for example in \cite{BCR}, \cite{Lauda-Pfeiffer1} and \cite{Antonio}. In this paper we study a $2$-dimensional open-closed cobordism category in which the $1$-manifolds and the cobordisms carry a map to a background space. More precisely, for a fixed background space $X$ and subspace $N \subseteq X$, we study the category $\C$ whose objects are oriented $1$-manifolds $S$ embedded in $\mathbb{R}^{\infty}$ and equipped with a map $(S,\partial S) \to (X,N)$ and whose morphisms are oriented open-closed cobordisms $F$ embedded in $\mathbb{R}^{\infty} \times [0,1]$ and equipped with a map $(F,\partial_{f} F) \to (X,N)$. There is a topology on $\C$ which we give in section \ref{definitions}.

In open-closed string theory the strings move in a background manifold $M$ and there is often a collection of submanifolds $\{N_{b} \subseteq M\}_{b \in \mathcal{B}}$ specified. These are called the $D$-branes and the movement of an open string is restricted by the requirement that its endpoints must move within certain of the $D$-branes. The category $\C$ is intended to be a model of open and closed strings moving in $X$ with the components of $N$ as the set of $D$-branes.

The operation of disjoint union gives an ill-defined multiplicative structure on $\C$. This structure is enough to ensure that $B\C$ is an infinite loop space as we show in section \ref{infloopsp}. The main result of this paper is to identify that infinite loop space in the case that $X$ is simply-connected and $N$ is discrete. For technical reasons, we are forced to work with a modified version of $\C$ which we denote by $\tCb$.

In the recent preprint \cite{Josh}, the author studies cobordisms categories of manifolds with corners in a more general setting, dealing with higher dimensions and tangent structures on the manifolds. The techniques used are very different to those here, and follow \cite{4author}. One point of note is that \cite{Josh} gives a calculation of $B\mathcal{C}_{\ast,\ast}$ which differs from our calculation of $B\tilde{\mathcal{C}}^{b}_{\ast,\ast}$.

\vsp
\noindent \textit{Summary of results:}
We fix a space $X$, a subspace $N$ and a basepoint $\ast \in N$. Our techniques lead us to consider moduli spaces
\[
\mathscr{S}(F;X,N)_{\ast}: = E\mathrm{Diff}_{\mathrm{oc}}^{+}(F) \times_{\mathrm{Diff}_{\mathrm{oc}}^{+}(F)} \mathrm{Map}(F,
\partial_{f}F; X,N)_{\ast}
\]
where $F$ is a connected open-closed cobordism with $\partial_{out}F=S^{1}$, $\mathrm{Diff}_{\mathrm{oc}}^{+}(F)$ is the group of orientation-preserving diffeomorphisms that fix $\partial_{in} F \cup \partial_{out}F$ pointwise and $\mathrm{Map}(F,
\partial_{f}F; X,N)_{\ast}$ is the space of maps which carry $\partial_{in} F \cup \partial_{out}F$ to the basepoint.

In section \ref{homstab} we show that the spaces $\mathscr{S}(F;X,N)_{\ast}$ have good homology stability properties. This is a consequence of the homology stability of mapping class groups with twisted coefficients which was proved in \cite{Ralph-Ib} and \cite{Ivanov}. We actually require some extensions of the original stability theorems and we establish these in section \ref{homstab}.

We also consider stabilized versions of the spaces $\mathscr{S}(F;X,N)_{\ast}$. Let $F^{1}$ be
obtained from $F$ by gluing a torus with three holes to the outgoing
boundary of $F$ (here the torus with three holes is thought of as an open-closed cobordism from $S^{1}$ to itself). So $F^{1}$ is an open-closed cobordism with
$\partial_{in}F^{1}=\partial_{in}F$, $\partial_{out}F^{1} =
\partial_{out}F = S^{1}$ and $F^{1}$ has one more window and one
greater genus than $F$. \label{inclusion}The inclusion $F \to F^{1}$ induces a map
\[
\mathscr{S}(F;X,N)_{\ast} \to \mathscr{S}(F^{1};X,N)_{\ast}
\]
given by extending diffeomorphisms by the identity and extending the maps to $X$ by $c_{\ast}$, the constant map to the basepoint. Let $F^{2}$ be obtained from $F^{1}$ by
gluing on another torus with three holes and so on. Define $\mathscr{S}^{\infty}_{\infty}(F;X,N)_{\ast}$ to be
\[
\mathrm{hocolim} \hspace{1mm}
(\mathscr{S}(F;X,N)_{\ast} \to \mathscr{S}(F^{1};X,N)_{\ast} \to \mathscr{S}(F^{2};X,N)_{\ast} \to
\ldots).
\]

Using the homology stability for the spaces $\mathscr{S}(F;X,N)_{\ast}$ and a generalised group completion theorem, we prove the following theorem in section \ref{gpcomp}.
\begin{theorem}
\label{firsttheorem}
Suppose that $X$ is simply-connected and $N$ is discrete. For any connected open-closed cobordism $F$ with $\partial_{out}F=S^{1}$ there is a homology isomorphism
\[
\mathbb{Z} \times \mathbb{Z} \times \mathscr{S}_{\infty}^{\infty}(F;X,N)_{\ast} \to \Omega B \tCb.
\]
\end{theorem}

When $F$ is an ordinary cobordism, i.e. $\partial_{in}F$ and $\partial_{out}F$ are
closed and $F$ has no free boundary, the subspace $N$ plays no role. In this case we will write the space $\mathscr{S}(F;X,N)_{\ast}$ simply as $\mathscr{S}(F;X)_{\ast}$.
For such $F$ we may also wish to stabilize with respect to genus
only. Let $F^{[1]}$ denote the surface obtained from $F$ by gluing on a torus with two holes and let $F^{[2]}$ be obtained from $F^{[1]}$ in the same fashion. Define
\[
\mathscr{S}_{\infty}(F;X)_{\ast} = \mathrm{hocolim} \hspace{1mm}
(\mathscr{S}(F;X)_{\ast} \to \mathscr{S}(F^{[1]};X)_{\ast} \to \mathscr{S}(F^{[2]};X)_{\ast} \to
\ldots).
\]
The spaces $\mathscr{S}(F;X)_{\ast}$ and $\mathscr{S}_{\infty}(F;X)_{\ast}$ were defined and studied in \cite{Ralph-Ib}.

For any space $Y$ let $Q(Y)$ denote the infinite loop space $\mathrm{colim}_{k \to \infty} \Omega^{k} \Sigma^{k} Y$ and let $Y_{+}$ denote the union of $Y$ and a disjoint basepoint. Using the homology stability for the spaces $\mathscr{S}(F;X,N)_{\ast}$ and splitting techniques analogous to those in \cite{Bod-Till}, in section \ref{splittingsection} we prove the following.
\begin{theorem}
\label{secondtheorem}
Suppose that $X$ is simply-connected and $N$ is discrete. Let $F$ be a connected open-closed cobordism from some closed, oriented $1$-manifold to
$S^{1}$ and let $\bar{F}$ be the ordinary cobordism obtained by gluing a
disc to each window of $F$. Then there is a homology isomorphism
\[
\mathbb{Z} \times \mathscr{S}^{\infty}_{\infty}(F;X,N)_{\ast} \to
\mathscr{S}_{\infty}(\bar{F};X)_{\ast} \times Q((BS^{1} \times N)_{+}).
\]
\end{theorem}

To state our main theorem we must first recall the definition of the spectrum $MTSO(2)$ studied in \cite{4author}, \cite{MadTill} and \cite{Madsen-Weiss}. Let $\mathrm{Gr}^{+}_{n}(\mathbb{R}^{2})$ be the Grassmannian of oriented $2$-planes in $\mathbb{R}^{n}$ and $\gamma_{n}$ be the canonical bundle over it. Let $\gamma_{n}^{\perp}$ denote the orthogonal complement of $\gamma_{n}$, consisting of pairs $(U,v)$ of an oriented $2$-plane $U$ in $\mathbb{R}^{n}$ and an $n$-vector $v$ perpendicular to it. The $n$th space in the spectrum $MTSO(2)$ is the Thom space $Th(\gamma_{n}^{\perp})$. There is a natural inclusion $\mathrm{Gr}_{2}^{+}(\mathbb{R}^{n}) \to \mathrm{Gr}_{2}^{+}(\mathbb{R}^{n+1})$ and we have that $\gamma^{\perp}_{n+1} |_{\mathrm{Gr}_{2}^{+}(\mathbb{R}^{n})} \cong \gamma_{n}^{\perp} \oplus \epsilon^{1}$ where $\epsilon^{1}$ is a trivial line bundle. Thus we obtain a map
\[
\Sigma MTSO(2)_{n}=\Sigma Th(\gamma_{n}^{\perp}) \simeq Th(\gamma_{n}^{\perp} \oplus \epsilon^{1}) \to Th(\gamma_{n+1})=MTSO(2)_{n+1}
\]
and this completes the definition of $MTSO(2)$.

It was shown in \cite{Ralph-Ib} that $\mathbb{Z} \times \mathscr{S}_{\infty}(\bar{F};X)_{\ast}$ has the same homology as the infinite loop space $\Omega^{\infty}(MTSO(2) \wedge X_{+})$. In section \ref{pfmainthm} we use this result together with Theorems \ref{firsttheorem} and \ref{secondtheorem} to prove our main theorem.
\begin{theorem}
When $X$ is simply-connected and $N$ is discrete there is a weak homotopy equivalence
\[
\Omega B \tCb \simeq \Omega^{\infty}(MTSO(2) \wedge X_{+}) \times Q((BS^{1} \times N)_{+}).
\]
\end{theorem}

\vsp
\noindent
\textit{Acknowledgements:} This work is part of my PhD thesis, supervised by Ulrike Tillmann. I'd like to thank her for her guidance throughout the project.

\section{Definition and properties of the category}

\subsection{Definitions}
\label{definitions}

In this section we give the full definition of $\C$ and define the modified category $\tCb$.

Let $S_{m,n}$ denote the union of $m$ oriented circles and $n$
oriented intervals and let $\mathrm{Emb} \hspace{0.5mm}(S_{m,n},\R^{k})$
denote the space of smooth embeddings, equipped with the Whitney
$C^{\infty}$-topology. Set
\[
\mathrm{Emb} \hspace{0.5mm}(S_{m,n}, \Rinf):= \mathrm{colim}_{k \to \infty}
\hspace{0.5mm}\mathrm{Emb} \hspace{0.5mm}(S_{m,n}, \R^{k}).
\]

Let $\mathrm{Diff}^{+}(S_{m,n})$ denote the group of orientation-preserving
diffeomorphisms of $S_{m,n}$. This group acts on the right of the embedding
space by precomposition. It also acts on the space $\mathrm{Map} \hspace{0.5mm} (S_{m,n}, \partial S_{m,n}; X,N)$. The
object space of $\C$ is defined to be the disjoint union of orbit
spaces
\[
\coprod_{m,n \geq 0} \mathrm{Emb} \hspace{0.5mm}(S_{m,n}, \Rinf)
\times_{\mathrm{Diff}^{+}(S_{m,n})} \mathrm{Map} \hspace{0.5mm}(S_{m,n}, \partial
S_{m,n}; X,N).
\]

We can think of an object as an embedded $1$-manifold with a map to the background space. We'll often write objects in the category as $\mathrm{Diff}^{+}$-equivalences classes $[e,f]$. When we have an object $\alpha=[e,f]$ and $e:S \to \mathbb{R}^{\infty}$ and $f:(S, \partial S) \to (X,N)$ we will say that $S$ is the \emph{underlying $1$-manifold} of $\alpha$.

Next we move on to defining the morphism space in the category. We require our cobordisms to have parametrized collars around their incoming and outgoing boundary. Explicitly, this means that if $F$ is a $2$-dimensional open-closed cobordism, we assume that $F$ is equipped with parametrizations
\[
c_{in}: [0,1) \times \partial_{in}F \to F \hspace{1em} \text{and} \hspace{1em} c_{out}:(-1,0] \times \partial_{out}F \to
F.
\]

If $\epsilon \in (0,1)$ and $F$ is an open-closed cobordism, $\mathrm{Diff}_{\epsilon}^{\ast}(F)$ denotes the group of
orientation-preserving diffeomorphisms of $F$ which preserve the
decomposition of $\partial F$ and restrict to product maps on the $\epsilon$-collars. We define
\[
\mathrm{Diff}^{\ast}(F):= \mathrm{colim}_{\epsilon \to 0}\mathrm{Diff}_{\epsilon}^{\ast}(F).
\]

Let $\mathrm{Emb}_{\epsilon}(F, \R^{n} \times [0,1])$ denote the space of
embeddings which embed the incoming boundary in $\R^{n} \times \{0\}$, the outgoing boundary in $\R^{n} \times \{1\}$
and embed the $\epsilon$-collars perpendicular to these walls.

Define
\[
\mathrm{Emb} \hspace{0.5mm}(F, \Rinf \times [0,1]):= \mathrm{colim}_{n \to \infty, \epsilon \to 0} \mathrm{Emb}_{\epsilon}(F, \R^{n} \times [0,1]).
\]

The morphism space of $\C$ is defined to be
\[
\coprod_{F} \mathrm{Emb} \hspace{0.5mm}(F, \Rinf \times [0,1])
\times_{\mathrm{Diff}^{\ast}(F)} \mathrm{Map} \hspace{0.5mm}(F,\partial_{f}F; X,N).
\]
In this disjoint union there is one open-closed cobordism from each
isomorphism class, two open-closed cobordisms $F, F'$ being
isomorphic if there is an orientation-preserving diffeomorphism $\theta: F \to F'$ with
$\theta(\partial_{in} F) =
\partial_{in} F'$ and $\theta(\partial_{out} F) =
\partial_{out}
F'$. 

We can think of a morphism as an embedded open-closed cobordism equipped with a map to the background space. We'll write morphisms as $\mathrm{Diff}^{\ast}$-equivalence classes $[h,\phi]$. When we have a morphism $m=[h,\phi]$ and $h:F \to \mathbb{R} \times [0,1]$ and $\phi:(F,\partial_{f}F) \to (X,N)$ we will say that $m$ has \emph{underlying open-closed cobordism} $F$.

Note that the assignment $(X,N) \mapsto \C$ defines a functor from
the category of pairs of spaces to the category of topological
categories. This functor is homotopy invariant i.e. a homotopy
equivalence of pairs $(X,N) \to (X',N')$ induces a homotopy
equivalence $B\C \to B\mathcal{C}_{X',N'}$.

Let us remark on two interesting choices of background space. In the case that $N=\emptyset$, there are no maps $(S, \partial S) \to (X,N)$ unless $\partial S$ is empty. Thus the only objects in $\C$ are closed $1$-manifolds. Similarly the only cobordisms are those with no free boundary. In this case $\C$ is the same as the ordinary $2$-dimensional oriented cobordism category with background space, $\mathcal{C}^{+}_{2}(X)$, as studied in \cite{4author}. When $X$ is contractible and $N$ is discrete, a map $(S, \partial S) \to (X,N)$ is the same thing as a labeling of the boundary points of $S$ by points in $N$ and a map $(F, \partial_{f}F) \to (X,N)$ is the same thing as a labeling of the components of $\partial_{f}F$ by points in $N$. In this case, $\C$ is a model of the open-closed cobordism category with labeling set $N$ as studied in \cite{BCR}.

We'll now move on to define the modified category $\tCb$. Let $\C^{b}$ denote the subcategory of $\C$ in which every component of every cobordism must have non-empty outgoing boundary. In $\C^{b}$, any morphism to a circle must be connected. The category $\tCb$ will be an enriched version of $\C^{b}$; its objects are like those of $\C^{b}$ except that they carry an extra piece of data and the morphisms must respect that data.

To explain this, let us first introduce a $1$-dimensional cobordism category $\mathcal{D}$. The objects in $\mathcal{D}$ are compact oriented $0$-dimensional manifolds. The morphisms in $\mathcal{D}$ are equivalence classes of oriented $1$-dimensional cobordisms. Two such cobordisms $L_{1}$ and $L_{2}$ are said to be equivalent if there is an orientation-preserving  diffeomorphism $I(L_{1}) \to I(L_{2})$, where $I(L_{i})$ denotes the union of the intervals in $L_{i}$, that restricts to the identity on the boundary. See Figure \ref{categoryD}.

\begin{figure}[h]
\begin{center}
\includegraphics{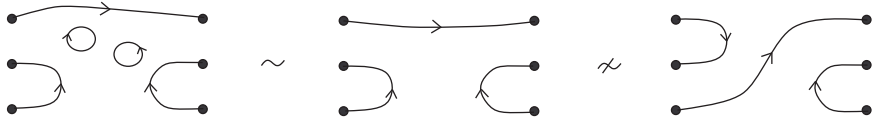}
\end{center}
\caption{Morphisms in the category $\mathcal{D}$
\label{categoryD}}
\end{figure}

There is a functor $\mathrm{R}: \C^{b} \to \mathcal{D}$ given by sending an object with underlying $1$-manifold $S$ to $\partial S$ and a morphism with underlying open-closed cobordism $F$ to the equivalence class of $\partial_{f}F$, see Figure \ref{functorR}. $\tCb$ is defined to be the Quillen over-category $\mathrm{R} / \emptyset$. Thus the objects of $\tCb$ are pairs $(\alpha, l)$ where $\alpha$ is an object in $\C^{b}$ and $l$ is a morphism in $\mathcal{D}$ from $R(\alpha) \to \emptyset $. A morphism from $(\alpha, l)$ to $(\beta,k)$ in $\tCb$ is a morphism $m:\alpha \to \beta$ in $\C^{b}$ such that $k \circ \mathrm{R}(m) = l$.

\begin{figure}[h]
\begin{center}
\includegraphics{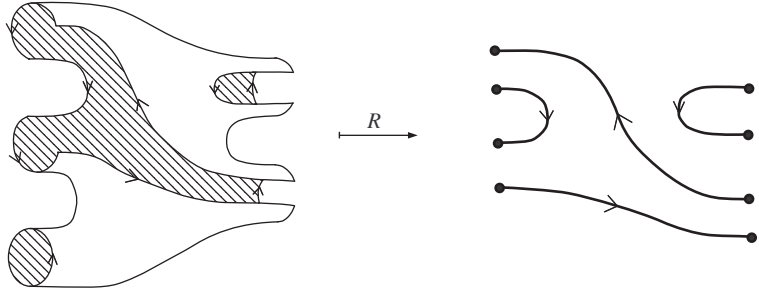}
\end{center}
\caption{The functor $R$}
\label{functorR}
\end{figure}

If $\alpha$ is an object in $\C^{b}$ whose underlying $1$-manifold is closed then there is only one morphism in $\mathcal{D}$ from $\mathrm{R}(\alpha)=\emptyset$ to $\emptyset$; the equivalence class of the empty cobordism. We will use $\varnothing$ to denote this morphism.

Note that if $\beta$ is an object in $\C^{b}$ whose underlying $1$-manifold is a circle then any morphism $m:(\alpha,l) \to (\beta, \varnothing)$ in $\tCb$ must be connected and satisfy $\varnothing \circ \mathrm{R}(m) = l$. If $F$ is the underlying cobordism of $m$ this means that there is an orientation-preserving diffeomorphism $I(\partial_{f}F) \to I(l)$ that restricts to the identity on the boundary. Thus the isomorphism type of $F$ is determined by only two invariants; the genus and the number of windows (the remaining boundary structure being prescribed by $l$.) This is the reason that we want to work with $\tCb$; in section \ref{gpcomp} we need to have good control over the space of morphisms from an arbitrary object to a circle.

\subsection{Path-Components}
\label{path-components}
When $N$ has more than one path-component, the classifying space $B\tCb$ is not
connected. In this section we will give a description of $\pi_{0}B\tCb$ in the case that $X$ is simply-connected.

For an oriented $1$-manifold $S$ let $\partial_{+}S \subseteq \partial S$ denote the collection of points that are positively-oriented in the orientation induced from $S$ and $\partial_{-}S$ those that are negatively-oriented.

\begin{definition}
The \emph{weighted signature} of an object $(\alpha,l)=([e,f],l)$ in $\tCb$ with underlying $1$-manifold $S$ is the function $b_{(\alpha,l)}:\pi_{0}N \to \mathbb{Z}$ given by
\[
b_{(\alpha,l)}(N_{i})=|\partial_{+}S \cap f^{-1}(N_{i})|-|\partial_{-}S \cap f^{-1}(N_{i})|.
\]
\end{definition}

This definition does not depend on the choice of representative $(e,f)$ for $\alpha$.

Note that
\begin{align*}
\text{(i)} & \sum_{N_{i} \in \pi_{0}N} b_{(\alpha,l)}(N_{i}) = |\partial_{+}S|-|\partial_{-}S|=0, \hspace{2mm}\\
\text{(ii)} & \sum_{N_{i} \in \pi_{0}N} |b_{(\alpha,l)}(N_{i})| \leq |\partial_{+}S| + |\partial_{-}S|<\infty,
\end{align*}
and as a consequence of (i),
\begin{align*}
\text{(iii)} \sum_{N_{i} \in \pi_{0}N} |b_{(\alpha,l)} (N_{i})| \hspace{1mm} \text{is always even}. \hspace{11mm}
\end{align*}
So each $b_{(\alpha,l)}$ is an element of
\[
\mathcal{B}_{N}:=\{b:\pi_{0}N \to \mathbb{Z}: \sum_{N_{i} \in \pi_{0}N} b(N_{i}) =0, \sum_{N_{i} \in \pi_{0}N} |b(N_{i})| \in 2 \mathbb{Z}\}.
\]
Furthermore, every element of $\mathcal{B}_{N}$ can occur as a weighted signature.

\begin{lemma}
\label{balpha=bbeta}
If there is a morphism $m:(\alpha,l) \to (\beta,k)$ in $\tCb$  then $b_{(\alpha,l)}=b_{(\beta,k)}$.
\end{lemma}

\begin{proof}
Let $\alpha=[e,f]$ and $\beta=[\tilde{e},\tilde{f}]$ have underlying $1$-manifolds $S,\tilde{S}$ respectively. Suppose that $m=[h,\phi]$ has underlying open-closed cobordism $F$. The path-components of $\partial_{f}F$ are either windows or intervals joining points of $\partial_{+}S \cup \partial_{-}\tilde{S}$ to points of $\partial_{-}S \cup \partial_{+}\tilde{S}$ (see Figure \ref{occob1} where points of $\partial_{+}S \cup \partial_{-}\tilde{S}$ are marked $\odot$ and points of $\partial_{-}S \cup \partial_{+}\tilde{S}$ are marked $\bullet$). If $L$ is such an interval then $\phi(l) \subseteq N$ is a path from $\phi(\partial_{-}L)$ to $\phi(\partial_{+}L)$. Hence if two points are joined by an interval in $\partial_{f}F$, they must map into the same path-component of $N$. Thus for each path-component $N_{i}$ we must have
\[
\begin{array}{ll}
|(\partial_{+}S \cup \partial_{-}\tilde{S}) \cap f^{-1}(N_{i})| \\
= |(\partial_{-}S \cup \partial_{+}\tilde{S}) \cap f^{-1}(N_{i})|
\end{array}
\]
which implies $b_{(\alpha,l)}=b_{(\beta,k)}$.
\end{proof}

\begin{figure}[h]
\begin{center}
\includegraphics{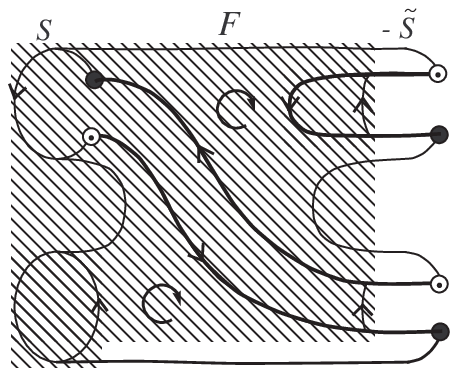}
\end{center}
\caption{Joining oriented points}
\label{occob1}
\end{figure}

\begin{corollary}
\label{decomposition}
The category $\tCb$ decomposes as the disjoint union
\[
\tCb = \coprod_{b \in \mathcal{B}_{N}} \tCb(b)
\]
where $\tCb(b)$ is the full-subcategory of $\tCb$ whose objects are those with weighted signature $b$.
\end{corollary}

\begin{proposition}
$\pi_{0} B\tCb \cong \mathcal{B}_{N}$.
\end{proposition}

\begin{proof}
Given Corollary \ref{decomposition} it suffices to show that each $B\tCb(b)$ is path-connected. To do this it is enough to show that for every pair of objects $(\alpha,l)$ and $(\beta,k)$ in $\tCb(b)$ there exists a third object $(\gamma,j)$ and morphisms $(\alpha,l) \to (\gamma,j)$ and $(\beta,k) \to (\gamma,j)$. The constructions below are illustrated in Figure \ref{alphabetagamma}.

\begin{figure}[t]
\begin{center}
\includegraphics{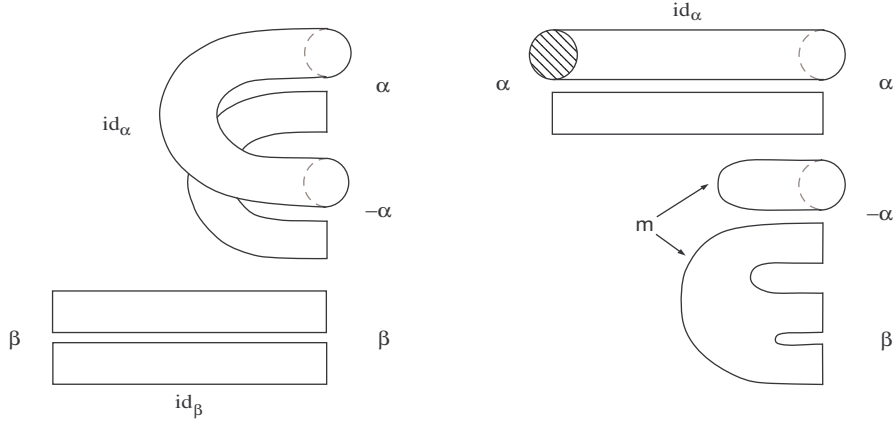}
\end{center}
\caption{Morphisms $\alpha \to \gamma$ and $\beta \to \gamma$}
\label{alphabetagamma}
\end{figure}

Let $-\alpha$ be the same as $\alpha$ but with the underlying $1$-manifold having the opposite orientation and define $(\gamma,j)=(\alpha \cup -\alpha \cup \beta, l \cup -l \cup k)$. Here we implicitly assume that the embeddings of the objects have been shifted so as to be disjoint. The identity morphism $\mathrm{id}_{\alpha}$ can be twisted around to give a morphism $\emptyset \to (\alpha \cup -\alpha,l \cup -l)$ and if we take the union of this twisted morphism with $\mathrm{id}_{\beta}$ we obtain a morphism
\[
(\beta,k) \to (\alpha \cup -\alpha \cup \beta, l \cup -l \cup k)=(\gamma,j).
\]

To construct a morphism $(\alpha,l) \to (\gamma,j)$ we first show that there is a morphism $m:\alpha \to \beta$ in $\C^{b}$. Suppose that $\alpha=[e,f]$, $\beta=[\tilde{e},\tilde{f}]$ and they have underlying $1$-manifolds $S$, $\tilde{S}$ respectively. Since $b_{(\alpha,l)} = b_{(\beta,k)}$ it is possible to pair the points of $\partial_{+}S \cup \partial_{-} \tilde{S}$ with the points of $\partial_{-}S \cup \partial_{+} \tilde{S}$ in such a way that two paired points map into the same component of $N$. If we attach an oriented interval between paired points we obtain a cobordism $S_{0}$ from $\partial S$ to $\partial \tilde{S}$.

The closed $1$-manifold $S \cup S_{0} \cup -\tilde{S}$ is the boundary of some connected, oriented open-closed cobordism $F$ from $S$
to $\tilde{S}$ (with $\partial_{f}F=S_{0}$). Furthermore there is an embedding $h: F \to \Rinf \times [0,1]$
with $h \vert_{\partial_{in}F} = e$ and $h
\vert_{\partial_{out}F} = \tilde{e}$. 

Because of the way that $S_{0}$ was constructed, the maps $f: (S,\partial S) \to (X,N)$ and $\tilde{f}:(\tilde{S}, \partial \tilde{S}) \to (X,N)$ can be extended over $S_{0}$ to give a map $\bar{f}: (\partial F,S_{0}) \to (X,N)$. Since $X$ is simply-connected, $\bar{f}$ can be extended to a map $\phi: (F, \partial_{f}F) \to (X,N)$. Now $m=[h,\phi]:\alpha \to \beta$ is our required morphism.

Note that $m$ does not necessarily give a morphism $(\alpha,l) \to (\beta,k)$ because we don't know that $k \circ \mathrm{R}(m)=l$. However, the morphism $m \cup \mathrm{id}_{\alpha}$ from $(\alpha \cup \alpha) \to (\alpha \cup \beta)$ can be twisted around to give the required morphism
\[
(\alpha,l) \to (\alpha \cup -\alpha \cup \beta, l \cup -l \cup k)=(\gamma,j).
\]
\end{proof}

Disjoint union in the category corresponds to addition of elements
in $\mathcal{B}_{N}$. Thus we have
\begin{corollary}
When $X$ is simply-connected, $\pi_{0}B\tCb$ is a group.
\end{corollary}

These methods can be adapted to show that $\pi_{0} B\C$ is in bijective correspondence with $\mathcal{B}_{N}$ and that it's a group.
 
\subsection{Infinite loop space structure}
\label{infloopsp}

Disjoint union gives a multiplication operation on $\tCb$ and, although this is not well-defined, it is enough to make $B\tCb$ into a $\Gamma$-space in the sense of \cite{Segal}. In the case that $X$ is simply-connected, we also know that $\pi_{0}B\tCb$ is a group and these two facts together imply that $B\tCb$ is an infinite loop space \cite{Segal}. We will now explain the $\Gamma$-space structure on $B\tCb$, closely following \cite{MadTill}, Theorem 2.3.

Let $\Gamma^{\mathrm{op}}$ denote the category of finite based sets $\underline{n} = \{0,1, \ldots, n\}$ and based maps. A $\Gamma$-space is a functor $\mathcal{F}$ from $\Gamma^{\mathrm{op}}$ to simplicial spaces satisfying
\begin{itemize}
\item[(i)] $\mathcal{F}(\underline{0}) \simeq \ast$.
\item[(ii)] The map $\prod_{i=1}^{n} p_{i}: \mathcal{F}(\underline{n}) \to \mathcal{F}(\underline{1}) \times \cdots \times \mathcal{F}(\underline{1})$ is a homotopy equivalence where $p_{i}:\underline{n} \to \underline{1}$ sends $i$ to $1$ and $j \neq i$ to $0$.
\end{itemize}

We will now define a $\Gamma$-space $\mathcal{N}$, with $\mathcal{N}(\underline{1})= \mathrm{N}_{\centerdot}\tCb$, the nerve of the category $\tCb$. Before we begin, let us introduce some notation; if $m=[h_{0}, \phi_{0}] \circ \cdots \circ [h_{q-1}, \phi_{q-1}]$ is an element of $\mathrm{N}_{q}\tCb$, i.e., a $q$-tuple of composable morphisms, we will use $\pi_{0}m$ to denote the connected components of $\mathrm{Im}(h_{0} \# \cdots \# h_{q-1})$. If the target of $m$ is the object $(\alpha,l)$, then we will say that a labeling function $\lambda: \pi_{0}m \to \{1, \ldots , n\}$ is \textit{good} if $l: \mathcal{R}(\alpha) \to \varnothing$ joins two points of $\mathcal{R}(\alpha)$ only if they come from elements of $\pi_{0}m$ with the same label.

Now let $\mathcal{N}(\underline{n})$ be the simplicial space whose $q$-simplices are the pairs $(m, \lambda)$ where $m \in \mathrm{N}_{q}\tCb$ and $\lambda: \pi_{0}m \to \{1, \ldots, n\}$ is a good labeling function. Given $s: \underline{n} \to \underline{m}$, define $s_{\ast}: \mathcal{N}(\underline{n}) \to \mathcal{N}(\underline{m})$ by $s_{\ast}(m, \lambda)=(\tilde{m}, s \circ \lambda)$ where $\tilde{m}$ is obtained from $m$ by deleting any components that are labeled zero by $s \circ \lambda$. 

\begin{proposition}
$\mathcal{N}$ is a $\Gamma$-space and hence, when $X$ is simply-connected, $B\tCb$ is an infinite loop space.
\end{proposition}

\begin{proof}(Following \cite{MadTill}.)
Certainly (i) is satisfied since $\mathcal{N}(0)$ is the nerve of the category containing only the empty $1$-manifold and the empty cobordism. To check (ii), note that the map $p_{i}$ sends $(m,\lambda) \in \mathcal{N}(n)_{q}$ to the union of the components of $m$ that are labeled $i$ by $\lambda$. Thus $\prod_{i=1}^{n} p_{i}$ carries $\mathcal{N}(n)_{q}$ injectively onto $\mathcal{N}(1)_{q} \times \cdots \times \mathcal{N}(1)_{q} \setminus \Delta$ where $\Delta$ is the fat diagonal of elements $((m_{1}, \lambda_{1}), \ldots, (m_{n},\lambda_{n}))$ such that $m_{i} \cap m_{j} \neq \emptyset$ for some $i \neq j$. The inclusion of the complement of the fat diagonal into $\mathcal{N}(1)_{q} \times \cdots \times \mathcal{N}(1)_{q}$ is a homotopy equivalence; this is a consequence of the fact that the embedding spaces are contractible, see \cite{MadTill} for details. Thus 
\[
\prod_{i=1}^{n} p_{i}: \mathcal{N}(\underline{n}) \to \mathcal{N}(\underline{1}) \times \ldots \times \mathcal{N}(\underline{1})
\]
 is the product of a homeomorphism and a homotopy equivalence, and hence is itself a homotopy equivalence.
\end{proof}

These methods can be adapted to show that $B\C$ is also an infinite loop space.

\section{The homotopy type of the category}

In this section we prove the main theorem of the paper, namely, we
identify the homotopy type of $B\tCb$ in the case that $X$ is
simply-connected and $N$ is discrete. We work under these
assumptions throughout although some of the methods apply more generally.

\subsection{Some requisite homology stability}
\label{homstab}

In this section we establish the homology stability of the spaces $\mathscr{S}(F;X,N)_{\ast}$ defined in the introduction. As mentioned there, this is a consequence of the homology stability of mapping class groups with twisted coefficients but we first need to establish some extensions of the original stability results.

We use $F_{g,n}$ to denote a surface of genus $g$ with $n$ boundary components and $\Gamma_{g,n}=\pi_{0} \mathrm{Diff}^{+}(F_{g,n}, \partial F_{g,n})$ to denote the associated mapping class group. A special case of stability
with coefficients is the following.

\begin{theorem}[\cite{Ralph-Ib}, Theorem 4]
\label{india}
Let $X$ be a simply-connected space and $\ast$ a
basepoint in $X$. Then the homology groups
\[
H_{q}(\Gamma_{g,n}; H_{r}(\mathrm{Map}(F_{g,n},\partial F_{g,n};X,\ast)))
\]
are independent of $g$ and $n$ provided that $g \geq 2q+r+2$.
\end{theorem}

In \cite{Ralph-Ib}, Theorem \ref{india} was used to prove homology
stability for the spaces $\mathscr{S}(F;X)_{\ast}$ defined in the introduction.

In the results that follow, `cylinder' refers to the product of any $1$-manifold with an interval. We remind the reader that, with respect to the figures in this paper, precomposing a cobordism $F$ with another cobordism means gluing another cobordism to the left of $F$.
\begin{theorem}[\cite{Ralph-Ib}, Theorem 3]
\label{desaru}
Suppose that $X$ is simply-connected. Let $F$ be an ordinary, connected cobordism with $\partial_{out}F=S^{1}$ and let $F^{\#}$ be a connected cobordism obtained from $F$
by precomposing with another ordinary cobordism. Then the map
$\mathscr{S}(F;X)_{\ast} \to \mathscr{S}(F^{\#};X)_{\ast}$ induced by extending
diffeomorphisms by the identity and extending maps to $X$ by $c_{\ast}$
induces an isomorphism
\[
H_{s}(\mathscr{S}(F;X)_{\ast}) \to H_{s}(\mathscr{S}(F^{\#};X)_{\ast})
\]
provided the genus of $F$ is greater than $2s+4$.
\end{theorem}

The following three theorems contain the extensions of Theorem \ref{india} that we'll need. The proofs are contained in the Appendix. The notation for the mapping spaces and diffeomorphism groups is as in the introduction. We use $\Gamma_{\mathrm{oc}}(F)$ to denote $\pi_{0}(\mathrm{Diff}_{\mathrm{oc}}^{+}(F))$ where $F$ is an open-closed cobordism. The first theorem describes the situation where the number of windows is fixed.

\begin{theorem}
\label{homstab1}
Suppose that $X$ is simply-connected and $N$ is discrete. Let $F$ be a connected open-closed cobordism with $\partial_{out}F=S^{1}$ and let $F^{\#}$ be a connected cobordism obtained from $F$ by precomposing with another cobordism. If $F$ and $F^{\#}$ have the same number of windows then the map
\begin{eqnarray*}
&& H_{s}(B\Gamma_{\mathrm{oc}}(F), H_{r}(\mathrm{Map}(F, \partial_{f}F;X,N)_{\ast})) \\
&& \hspace{8mm} \longrightarrow H_{s}(B\Gamma_{\mathrm{oc}}(F^{\#}), H_{r}(\mathrm{Map}(F^{\#}, \partial_{f}F^{\#};X,N)_{\ast}))
\end{eqnarray*}
is an isomorphism provided the genus of $F$ is at least $2s+r+2$.
\end{theorem}

The next two results tell us what happens when we increase the number of windows.

\begin{theorem}
\label{homstab2}
Suppose that $X$ is simply-connected and $N$ is a point. Let $F$ be an open-closed cobordism as in Theorem \ref{homstab1} and let $F^{\#}$ be obtained from $F$ by precomposing with a cylinder with one window. Then the map
\begin{eqnarray*}
&& H_{s}(B\Gamma_{\mathrm{oc}}(F), H_{r}(\mathrm{Map}(F, \partial_{f}F;X,N)_{\ast})) \\
&& \hspace{8mm} \longrightarrow H_{s}(B\Gamma_{\mathrm{oc}}(F^{\#}), H_{r}(\mathrm{Map}(F^{\#}, \partial_{f}F^{\#};X,N)_{\ast}))
\end{eqnarray*}
is an isomorphism provided the genus of $F$ is greater than $2s+r+2$ and the number of windows is greater than $2s+8$.
\end{theorem}

When $N$ is discrete but contains more than one point the proof of Theorem \ref{homstab2} fails; we do not know that adding an extra window induces an isomorphism on the homology
\[
H_{s}(B\Gamma_{\mathrm{oc}}(F), H_{r}(\mathrm{Map}(F, \partial_{f}F;X,N)_{\ast})).
\]
Instead we find that in the limit, as the genus and number of windows tend to infinity, we get a homology isomorphism.

The inclusion $F \to F^{1}$ from the introduction induces a map $\Gamma_{\mathrm{oc}}(F) \to \Gamma_{\mathrm{oc}}(F^{1})$ by extending diffeomorphisms by the identity. Let
\[
\Gamma^{\infty}_{\infty}(F)=\mathrm{colim} \hspace{1mm} (\Gamma_{\mathrm{oc}}(F) \to \Gamma_{\mathrm{oc}}(F^{1}) \to \Gamma_{\mathrm{oc}}(F^{2}) \to \ldots).
\]
We also have a map
\[
\mathrm{Map}(F,\partial_{f}F;X,N)_{\ast} \to \mathrm{Map}(F^{1},\partial_{f}F^{1};X,N)_{\ast}
\]
given by extending maps by $c_{\ast}$, the constant map to the basepoint. Define $\mathrm{Map}_{\infty}^{\infty}(F,\partial_{f}F;X,N)_{\ast}$ to be
\[
\mathrm{hocolim} \hspace{1mm} (\mathrm{Map}(F,\partial_{f}F;X,N)_{\ast} \to \mathrm{Map}(F^{1},\partial_{f}F^{1};X,N)_{\ast} \to \ldots).
\]
For each $k$ there is an action of $\Gamma_{\mathrm{oc}}(F^{k})$ on $\mathrm{Map}(F^{k},\partial_{f} F^{k};X,N)_{\ast}$ and these actions are compatible with the maps used to form the limits. Thus we have an action of $\Gamma^{\infty}_{\infty}(F)$ on $H_{s}(\mathrm{Map}_{\infty}^{\infty}(F,\partial_{f}F;X,N)_{\ast})$ and we can form the homology group
\[
H_{r}(B\Gamma_{\infty}^{\infty}(F);H_{s}(\mathrm{Map}_{\infty}^{\infty}(F,\partial_{f}F;X,N)_{\ast})).
\]

If $F^{\#}$ is obtained from $F$ by precomposing with some other open-closed cobordism then we get induced maps
\begin{eqnarray*}
\Gamma_{\infty}^{\infty}(F) & \to & \Gamma_{\infty}^{\infty}(F^{\#}) \\
\mathrm{Map}_{\infty}^{\infty}(F,\partial_{f}F;X,N)_{\ast} & \to & \mathrm{Map}_{\infty}^{\infty}(F^{\#},\partial_{f}F^{\#};X,N)_{\ast}
\end{eqnarray*}
and an induced map on the twisted homology.

\begin{theorem}
\label{homstab3}
Suppose that $X$ is simply-connected and $N$ is discrete. Let $F$ be as in Theorem \ref{homstab1} and let $F^{\#}$ be obtained from $F$ by precomposing with a cylinder with one window. Then the induced map
\begin{eqnarray*}
&& H_{s}(B\Gamma_{\infty}^{\infty}(F);H_{r}(\mathrm{Map}_{\infty}^{\infty}(F,\partial_{f}F;X,N)_{\ast})) \\
&& \hspace{8mm} \longrightarrow H_{s}(B\Gamma_{\infty}^{\infty}(F^{\#});H_{r}(\mathrm{Map}_{\infty}^{\infty}(F^{\#},\partial_{f}F^{\#};X,N)_{\ast}))
\end{eqnarray*}
is an isomorphism for all $s,r$.
\end{theorem}

We'll now use these theorems to prove homology stability properties for the spaces $\mathscr{S}(F;X,N)_{\ast}$.

\begin{corollary}[To Theorem \ref{homstab1}]
\label{surfstab1}
Suppose that $X$ is simply-connected and $N$ is discrete. Let $F$ be as in Theorem \ref{homstab1} and let $F^{\#}$ be a connected open-closed cobordism obtained from $F$ by precomposing with another cobordism. If $F$ and $F^{\#}$ have the same number of windows then inclusion induces an isomorphism
\[
H_{s}(\mathscr{S}(F;X,N)_{\ast}) \to H_{s}(\mathscr{S}(F^{\#};X,N)_{\ast})
\]
provided the genus of $F$ is greater than $2s+4$.
\end{corollary}

\begin{proof}
Consider the following map of fibrations
\begin{eqnarray*}
\begin{diagram}
\node{\mathrm{Map}(F,\partial_{f}F;X,N)_{\ast}}  \arrow{e}  \arrow{s}
\node{\mathscr{S}(F;X,N)_{\ast}} \arrow{e} \arrow{s}
\node{B\mathrm{Diff}_{\mathrm{oc}}^{+}(F)} \arrow{s}
\\
\node{\mathrm{Map}(F^{\#},\partial_{f}F^{\#};X,N)_{\ast}} \arrow{e}
\node{\mathscr{S}(F^{\#};X,N)_{\ast}} \arrow{e}
\node{B\mathrm{Diff}_{\mathrm{oc}}^{+}(F^{\#})}
\end{diagram}
\end{eqnarray*}
As a consequence of \cite{EE,ES}, the components of the diffeomorphism groups are contractible except in a few low genus cases so $B\mathrm{Diff}_{\mathrm{oc}}^{+}(F) \simeq B\Gamma_{\mathrm{oc}}(F)$. There is an induced map of the Serre spectral sequences associated to these fibrations and on the $E^{2}$-page it is the map
\begin{eqnarray*}
&& E_{p,q}^{2}=H_{p}(\Gamma_{\mathrm{oc}}(F),H_{q}(\mathrm{Map}(F, \partial_{f}F;X,N)_{\ast})) \\
&& \hspace{8mm} \longrightarrow \tilde{E}_{p,q}^{2}=H_{p}(\Gamma_{\mathrm{oc}}(F^{\#}),H_{q}(\mathrm{Map}(F^{\#} \partial_{f}F^{\#};X,N)_{\ast})).
\end{eqnarray*}
Theorem \ref{homstab1} tells us this map is an isomorphism when the genus of $F$ is greater than $2p+q+2$ and the corollary follows by the Zeeman comparison theorem (see \cite{Ivanov}, Theorem 1.2, for the comparison theorem in precisely the form we are using).
\end{proof}

\begin{corollary}[To Theorem \ref{homstab2}]
\label{surfstab2}
Suppose $X$ is simply-connected and $N$ is a point. Let $F$ be as in Theorem \ref{homstab1} and let $F^{\#}$ be obtained from $F$ by precomposing with a cylinder with one window then
\[
H_{s}(\mathscr{S}(F;X,N)_{\ast}) \to H_{s}(\mathscr{S}(F^{\#};X,N)_{\ast})
\]
is an isomorphism provided the genus of $F$ is greater than $2s+4$
and the number of windows is greater than $2s+10$.
\end{corollary}

\begin{proof}
The proof is analogous to that of Corollary \ref{surfstab1}.
\end{proof}

\begin{corollary}[To Theorem \ref{homstab3}]
\label{surfstab3}
Suppose that $X$ is simply-connected and $N$ is discrete. Let $F$ be as in Theorem \ref{homstab1} and $F^{\#}$ be obtained from $F$ by precomposing with a cylinder with one window. Then the map
\[
H_{s}(\mathscr{S}_{\infty}^{\infty}(F;X,N)_{\ast}) \to H_{s}(\mathscr{S}_{\infty}^{\infty}(F^{\#};X,N)_{\ast})
\]
is an isomorphism for all $s$.
\end{corollary}

\begin{proof}
Again, this is analogous to the proof of Corollary \ref{surfstab1}.
\end{proof}

\subsection{A group completion argument}
\label{gpcomp}

Given a contravariant functor $F$ from a topological category $\mathcal{C}$ to $\mathrm{Spaces}$ we can consider the category $F \wr \mathcal{C}$. The objects of $F \wr \mathcal{C}$ are pairs $(c,x)$ with $c \in \mathrm{Obj} \hspace{1mm} \mathcal{C}$ and $x \in F(c)$. The morphisms are pairs $(m,y)$ with $m:a \to b$ a morphism in $\mathcal{C}$ and $x \in F(b)$; the morphism $(m,y)$ has source $(a, F(m)(x))$ and target $(b,x)$. There is a natural projection functor $F\wr \mathcal{C} \to \mathcal{C}$.

\begin{theorem}[\cite{4author}, Proposition 7.1]
\label{gp compl}
Suppose that the natural map $\mathrm{Obj} (F \wr \mathcal{C}) \to \mathrm{Obj} \hspace{1mm} \mathcal{C}$ is a Serre fibration and that $B(F \wr \mathcal{C})$ is contractible. If every morphism $m:a \to b$ in $\mathcal{C}$ induces an isomorphism
\[
F(m)_{\ast}:H_{\ast}(F(\alpha); \mathbb{Z}) \to H_{\ast}(F(\beta);\mathbb{Z})
\]
then for each object $c \in \mathcal{C}$ there is a map $F(c) \to \Omega_{c}B\mathcal{C}$ that induces an isomorphism in integral homology.
\end{theorem}
\noindent
Here $\Omega_{c}B\mathcal{C}$ denotes the space of loops in $B\mathcal{C}$ based at $c$.

We now define a functor $\mathrm{Mor}_{\infty}: \tCb \to \mathrm{Spaces}$ and show that Theorem \ref{gp compl} applies to it. Let $\beta_{0}=[e_{0}, c_{\ast}]$ in $\mathrm{Obj} \hspace{1mm} \C^{b}$ where $e_{0}:S^{1} \to \mathbb{R}^{2} \to \mathbb{R}^{\infty}$ is the canonical embedding and $c_{\ast}:S^{1} \to X$ is the map which is constant at the basepoint. Define a contravariant functor $\mathrm{Mor}_{0}:\tCb \to \mathrm{Spaces}$ by
\[
\mathrm{Mor}_{0}((\alpha,l)):=\mathrm{Mor}_{\tCb}((\alpha,l), (\beta_{0},\varnothing)).
\]

Note that $(\beta_{0},\varnothing)$ lies in the component $\tCb(0)$ of $\tCb$. For those objects $(\alpha,l)$ that lie in some other component, we know from Section \ref{path-components} that there are no morphisms from $(\alpha,l)$ to $(\beta_{0},\varnothing)$ so $\mathrm{Mor}_{0}((\alpha,l))$ is empty. For $(\alpha,l)$ in $\tCb (0)$ we have the following proposition.

\begin{proposition}
\label{morphism spaces}
Suppose $X$ is simply-connected and $N$ is discrete. Let $(\alpha,l)$ be an object in $\tCb (0)$ with underlying $1$-manifold $S$. There is a homotopy equivalence
\[
\mathrm{Mor}_{0}((\alpha,l)) \simeq \coprod_{g \geq 0} \coprod_{w \geq 0} \mathscr{S}(F^{w}_{g};X,N)_{\ast}
\]
where $F^{w}_{g}$ is a connected open-closed cobordism from $S$ to $S^{1}$ with $w$ windows, genus $g$ and having the boundary structure prescribed by $l$ (see section \ref{definitions}). 
\end{proposition}

\begin{proof}
Let $\alpha=[e,f]$ with underlying $1$-manifold $S$. For any open-closed cobordism $F$ from $S$ to $S^{1}$ let
\[
\left( \mathrm{Emb}(F;\mathbb{R}^{\infty} \times [0,1]) \times \mathrm{Map}(F, \partial_{f} F;X,N) \right)_{[e\cup e_{0},f\cup c_{\ast}]}
\]
be the subspace of pairs $(h,\phi)$ such that $[h|_{\partial_{in}F},\phi|_{\partial_{in}F}]=[e,f]$ and \linebreak $[h|_{\partial_{out}F},\phi|_{\partial_{out}F}]=[e_{0},c_{\ast}]$. Let
\[
\left( \mathrm{Emb}(F;\mathbb{R}^{\infty} \times [0,1]) \times \mathrm{Map}(F, \partial_{f} F;X,N) \right)_{e\cup e_{0},f\cup c_{\ast}}
\]
be the subspace of pairs $(h,\phi)$ such that $(h|_{\partial_{in}F},\phi|_{\partial_{in}F})=(e,f)$ and \linebreak $(h|_{\partial_{out}F},\phi|_{\partial_{out}F})=(e_{0},c_{\ast})$.

We noted in section \ref{definitions} that the isomorphism type of the cobordism underlying a morphism $m:(\alpha,l) \to (\beta_{0},\varnothing)$ is completely determined by its genus and the number of windows. So $\mathrm{Mor}_{\tCb}((\alpha,l), (\beta_{0},\varnothing))$ is given by
\begin{equation*}
\coprod_{g,w \geq 0} \left( \mathrm{Emb}(F^{w}_{g};\mathbb{R}^{\infty} \times [0,1]) \times \mathrm{Map}(F^{w}_{g}, \partial_{f} F^{w}_{g};X,N) \right)_{[e\cup e_{0},f\cup c_{\ast}]} / \mathrm{Diff}^{\ast}(F^{w}_{g}).
\end{equation*}
Using an argument entirely analogous to that in \cite{MadTill}, Theorem 2.1, it can be shown that the $(g,w)$ component is homotopy equivalent to
\[
\left( \mathrm{Emb}(F^{w}_{g};\mathbb{R}^{\infty} \times [0,1]) \times \mathrm{Map}(F^{w}_{g}, \partial_{f} F^{w}_{g};X,N) \right)_{e\cup e_{0},f\cup c_{\ast}} / \mathrm{Diff}_{\mathrm{oc}}^{+}(F^{w}_{g}).
\]
Using the fact that the embedding space is contractible (by the Whitney embedding theorems), this is homotopy equivalent to 
\begin{eqnarray}
\label{component}
E\mathrm{Diff}^{+}_{\mathrm{oc}}(F^{w}_{g}, \mathbb{R}^{\infty} \times [0,1]) \times_{\mathrm{Diff}^{+}_{\mathrm{oc}}(F^{w}_{g})} \mathrm{Map}(F^{w}_{g}, \partial_{f} F^{w}_{g};X,N)_{f\cup c_{\ast}}.
\end{eqnarray}
To complete the proof, we will display a homotopy equivalence 
\[
\mathscr{H}: \mathrm{Map}(F, \partial_{f}F; X,N)_{f \cup c_{\ast}} \to \mathrm{Map}(F, \partial_{f}F; X,N)_{c_{\ast} \cup c_{\ast}}.
\]
which induces a homotopy equivalence between the space in (\ref{component}), and $\mathscr{S}(F^{w}_{g};X,N)_{\ast}$.

For ease of notation, write $F=F^{w}_{g}$. To define $\mathscr{H}$, we first assume that for each non-window boundary component $\partial_{i}F$ of $F$, a closed collar 
$$
[0,1] \times \partial_{i}F \hookrightarrow F 
$$
has been specified, with $\{1\} \times \partial_{i}F$ identified with $\partial_{i}F$. Let $\phi$ be an element of $\mathrm{Map}(F, \partial_{f}F; X,N)_{f \cup c_{\ast}}$. Because $N$ is discrete, $f$ completely determines $\phi \vert_{\partial_{i}F}$ for every non-window boundary component of $F$. Since $X$ is simply-connected, each $\phi \vert_{\partial_{i}F}$ is null-homotopic, by a homotopy $H^{i}: S^{1} \times I \to X$ say. Define $\mathscr{H}$ as follows; glue a cylinder to each of the boundary components $\partial_{i}F$ and let $F'$ denote the surface we obtain. Given $\psi \in \mathrm{Map}(F, \partial_{f}F; X,N)_{f \cup c_{\ast}}$, define $\psi': F' \to X$ by setting $\psi'$ equal to $\psi$ on $F \subseteq F'$ and equal to $H^{i}$ on the cylinder glued on to $\partial_{i}F$. Fix a diffeomorphism $d:F \to F'$ that is the identity on the complement of the open collars $(0,1] \times \partial_{i}F$ and stretches the collars along the new cylinders. Set $\mathscr{H}(\psi):= \psi' \circ d$. This is a homotopy equivalence with homotopy inverse given by gluing on cylinders and extending maps by the conjugate homotopies $\bar{H^{i}}$.

The map $\mathscr{H}$ is not $\mathrm{Diff}^{+}_{\mathrm{oc}}(F)$-equivariant but we do have the following equivariance property. Let $\widetilde{\mathrm{Diff}}^{+}_{\mathrm{oc}}(F) \subseteq \mathrm{Diff}^{+}_{\mathrm{oc}}(F)$ denote the group of diffeomorphisms that restrict to the identity on the closed collars. The map $\mathscr{H}$ is $\widetilde{\mathrm{Diff}}^{+}_{\mathrm{oc}}(F)$-equivariant and the inclusion $\widetilde{\mathrm{Diff}}^{+}_{\mathrm{oc}}(F) \hookrightarrow \mathrm{Diff}^{+}_{\mathrm{oc}}(F)$ is a homotopy equivalence. Thus $\mathscr{H}$ induces a homotopy equivalence from the space in (\ref{component}) to $\mathscr{S}(F^{w}_{g};X,N)_{\ast}$, as required.
\end{proof}

Let $t:\beta_{0} \to \beta_{0}$ denote the morphism $[h_{0}, c_{\ast}]$ where $h_{0}:T^{2} \setminus (3 \hspace{2mm} \text{holes}) \to \mathbb{R}^{\infty} \times [0,1]$ is an embedding restricting to $e_{0}$ on the incoming and outgoing boundary (the torus with three holes being considered as a cobordism from the circle to itself) and $c_{\ast}$ again represents a constant map. For an object $(\alpha,l)$ in $\tCb$ define 
\[
\mathrm{Mor}_{\infty}((\alpha,l)) = \mathrm{hocolim} \hspace{1mm} (\mathrm{Mor}_{0}((\alpha,l)) \stackrel{T}{\longrightarrow} \mathrm{Mor}_{0}((\alpha,l))\stackrel{T}{\longrightarrow} \ldots)
\]
where the map $T$ is given by composing with $t$.
By Proposition \ref{morphism spaces} we have
\begin{equation}
\label{morinfty}
\mathrm{Mor}_{\infty}((\alpha,l)) \simeq \mathbb{Z} \times \mathbb{Z} \times \mathscr{S}_{\infty}^{\infty}(F_{(\alpha,l)};X,N)_{\ast}
\end{equation}
where $F_{(\alpha,l)}$ is some open-closed cobordism to the circle having the boundary structure prescribed by $(\alpha,l)$.

Precomposition with a morphism $m: (\alpha,l) \to (\beta,k)$ induces a map $\mathrm{Mor}_{0}((\beta,k)) \to \mathrm{Mor}_{0}((\alpha,l))$ which commutes with $T$. Thus we obtain a map $\mathrm{Mor}_{\infty}((\beta,k)) \to \mathrm{Mor}_{\infty}((\alpha,l))$ and we have that $\mathrm{Mor}_{\infty}$ is a contravariant functor from $\tCb$ to $\mathrm{Spaces}$.

\begin{proposition}
\label{brampton}
Suppose that $X$ is simply-connected and $N$ is discrete. If $m:(\alpha,l) \to (\beta,k)$ is a morphism in $\tCb$ whose map to the background space is constant at the basepoint then
\[
\mathrm{Mor}_{\infty}(m):\mathrm{Mor}_{\infty}((\beta,k)) \to \mathrm{Mor}_{\infty}((\alpha,l))
\]
is an isomorphism on integral homology.
\end{proposition}

\begin{proof}
In this case $(\alpha,l), (\beta,k)$ both lie in $\tCb(0)$ and we can think of $\mathrm{Mor}_{\infty}(m)$ as a map
\[
\mathrm{Mor}_{\infty}(m):\mathbb{Z} \times \mathbb{Z} \times \mathscr{S}_{\infty}^{\infty}(F_{(\beta,k)};X,N)_{\ast} \to \mathbb{Z} \times \mathbb{Z} \times \mathscr{S}_{\infty}^{\infty}(F_{(\alpha,l)};X,N)_{\ast}.
\]
This is an isomorphism by Corollaries \ref{surfstab1}, \ref{surfstab2} and \ref{surfstab3}.
\end{proof}

For morphisms whose underlying map is non-constant we must do a little more work. This is because our stability theorems only tell us what happens when we precompose with a cobordism and extend the maps to the background space by constant ones.

\begin{proposition}
\label{panuba}
When $X$ is simply-connected and $N$ is discrete, every morphism $m:(\alpha,l) \to (\beta,k)$ in $\tCb$ becomes an isomorphism in integral homology after applying $\mathrm{Mor}_{\infty}$.
\end{proposition}

\begin{proof}
If $(\alpha,l)$ and $(\beta,k)$ are in some component of $\tCb$ other than $\tCb(0)$ then $\mathrm{Mor}_{\infty}((\alpha,l))$ and $\mathrm{Mor}_{\infty}((\beta,k))$ are both empty so the conclusion of the proposition holds. Assume that $(\alpha,l)$ and $(\beta,k)$ are in $\tCb(0)$.

Using the fact that $X$ is simply-connected and $N$ is discrete, we can decompose $m$ as $m_{1} \circ m_{2} \circ \ldots \circ m_{k}$ where for each $i$ we have either
\begin{itemize}
\item[(a)] $m_{i}=[h_{i}, c_{i}]$ where $c_{i}$ is constant at some point in $N$
\item[(b)] the cobordism underlying $m_{i}$ is of the form $S \times I$ for some $1$-manifold $S$.
\end{itemize}
Those $m_{i}$ satisfying (a) induce homology isomorphisms on $\mathrm{Mor}_{\infty}$ by Proposition \ref{brampton}. Those $m_{i}=[h_{i}, \psi_{i}]$ satisfying (b) induce homotopy equivalences on $\mathrm{Mor}_{\infty}$ with homotopy inverse induced by $\bar{m}_{i}=[h_{i}, \bar{\psi}_{i}]$ where $\bar{\psi}_{i}(s,t)=\psi_{i}(s,1-t)$ for $(s,t) \in S \times I$.
\end{proof}

\begin{proposition}
\label{Serre fibration}
The projection map $\mathrm{Obj} \hspace{1mm}(\mathrm{Mor}_{\infty} \wr \tCb) \to \mathrm{Obj} \hspace{1mm} \tCb$ is a Serre fibration.
\end{proposition}

\begin{proof}
Consider the following pullback square
\begin{eqnarray*}
\begin{diagram}
\node{\mathrm{Obj} \hspace{1mm} (\mathrm{Mor}_{0} \wr \tCb)} \arrow{e} \arrow{s} \node{\mathrm{Mor} \hspace{1mm} \tCb} \arrow{s,r}{(\text{source},\text{target})} \\
\node{\mathrm{Obj} \hspace{1mm} \tCb} \arrow{e} \node{\mathrm{Obj} \hspace{1mm} \tCb \times \mathrm{Obj} \hspace{1mm} \tCb}.
\end{diagram}
\end{eqnarray*}
The right-hand map is given by restricting embeddings, and maps to the background space, to the boundary and hence is a Serre fibration. Thus the left-hand map and also $\mathrm{Obj} \hspace{1mm} (\mathrm{Mor}_{\infty} \wr \tCb) \to \mathrm{Obj} \hspace{1mm} (\tCb)$ are Serre fibrations.
\end{proof}

\begin{lemma}
\label{lemma}
$B(\mathrm{Mor}_{\infty} \wr \tCb)$ is contractible.
\end{lemma}

\begin{proof}
$B(\mathrm{Mor}_{\infty} \wr \tCb) \simeq \mathrm{hocolim} \hspace{1mm} B(\mathrm{Mor}_{1} \wr \tCb)$ but $B(\mathrm{Mor}_{1} \wr \tCb)$ is contractible since $\mathrm{Mor}_{1} \wr \tCb$ has a terminal object $(\beta_{0}, id_{\beta_{0}})$.
\end{proof}

\begin{theorem}
\label{malaysia}
When $X$ is simply-connected and $N$ is discrete there is a homology isomorphism
\[
f:\mathbb{Z} \times \mathbb{Z} \times \mathscr{S}_{\infty}^{\infty}(F_{(\alpha,l)};X,N)_{\ast} \to \Omega B \tCb
\]
for each $(\alpha,l) \in \mathrm{Obj} \hspace{1mm} \tCb(0)$.
\end{theorem}
\begin{proof}
By Propositions \ref{panuba} and \ref{Serre fibration} and Lemma \ref{lemma}, the conditions of Theorem \ref{gp compl} are satisfied and we can conclude that there is a homology isomorphism $Mor_{\infty}((\alpha,l)) \to \Omega B \tCb$. The theorem now follows from equation (\ref{morinfty}).
\end{proof}

\noindent
N.B. We have omitted the basepoint of the loop space from the notation. Since $B\tCb$ is an infinite loop space, all of its path-components are homotopy equivalent so it doesn't matter where we base the loops.

\begin{theorem}
\label{thailand}
Let $\mathcal{C}^{b,c}_{X,N}$ denote the full subcategory of $\tCb$ whose objects are closed $1$-manifolds. When $X$ is simply-connected and $N$ is discrete, the inclusion functor induces a weak homotopy equivalence
\[
\Omega B \C^{b,c} \to \Omega B \tCb.
\]
\end{theorem}

\begin{proof}
We could equally well apply the methods above to yield a homology isomorphism
\[
f':\mathbb{Z} \times \mathbb{Z} \times \mathscr{S}_{\infty}^{\infty}(F_{(\alpha,\varnothing)};X,N)_{\ast} \to \Omega B \C^{b,c}
\]
for each $(\alpha,\varnothing) \in \mathrm{Obj} \hspace{1mm} \C^{b,c} \subset \mathrm{Obj} \hspace{1mm}\tCb$.
The map induced by the inclusion functor fits into the following commutative diagram
\begin{eqnarray*}
\begin{diagram}
\node{\mathbb{Z} \times \mathbb{Z} \times \mathscr{S}_{\infty}^{\infty}(F_{(\alpha,\varnothing)};X,N)_{\ast}} \arrow{se,b}{f'} \arrow{e,t}{f} \node{\Omega B \C^{b,c}} \arrow{s}\\
\node{} \node{\Omega B \tCb \hspace{1mm}.}
\end{diagram}
\end{eqnarray*}
Since $f$ and $f'$ are homology isomorphisms, so is the vertical map and since its source and target are both loop spaces, it is a weak homotopy equivalence. 
\end{proof}

\subsection{A splitting}
\label{splittingsection}

\begin{theorem}
\label{splitting}
Suppose that $X$ is simply-connected and $N$ is discrete. Let $F$ be a connected open-closed cobordism from some closed, oriented $1$-manifold to
$S^{1}$ and let $\bar{F}$ be the ordinary cobordism obtained by gluing a
disc to each window of $F$ (see Figure \ref{occobs}). Then there is a homology isomorphism
\[
g:\mathbb{Z} \times \mathscr{S}^{\infty}_{\infty}(F;X,N)_{\ast} \to
\mathscr{S}_{\infty}(\bar{F};X)_{\ast} \times Q((BS^{1} \times N)_{+}).
\]
\end{theorem}

\begin{proof}
Let $F'$ be the ordinary cobordism that is homeomorphic to $F$ as a surface and has $\partial_{out}F'=S^{1}$ (see Figure \ref{occobs}).

\begin{figure}[t]
\begin{center}
\includegraphics{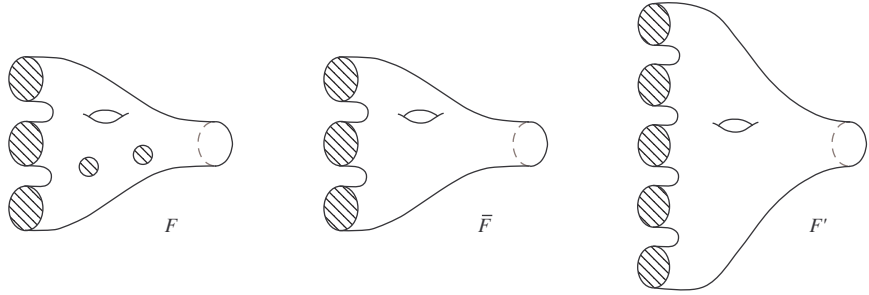}
\end{center}
\caption{The surfaces $F$, $\bar{F}$ and $F'$}
\label{occobs}
\end{figure}

Restriction to the windows yields a fibration
\begin{eqnarray*}
&& E\mathrm{Diff}_{\mathrm{oc}}^{+}(F) \times_{\mathrm{Diff}_{\mathrm{oc}}^{+}(F)} \mathrm{Map}(F,\partial_{f} F;X,N)_{\ast} \\
&& \hspace{2cm} \to E\mathrm{Diff}^{+}(\partial_{f} F) \times_{\mathrm{Diff}^{+}(\partial_{f}F)} \mathrm{Map}(\partial_{f}F;N)
\end{eqnarray*}
with fiber
\[
E\mathrm{Diff}_{\mathrm{oc}}^{+}(F') \times_{\mathrm{Diff}_{\mathrm{oc}}^{+}(F')} \mathrm{Map}(F',\partial F;X,\ast).
\]
Using the fact that $\mathrm{Diff}^{+}(\coprod_{w}S^{1}) \simeq \Sigma_{w} \wr S^{1}$ we see that this fibration is (homotopy equivalent to)
\begin{equation}
\label{fibration}
\mathscr{S}(F';X)_{\ast} \to \mathscr{S}(F;X,N)_{\ast} \to E(\Sigma_{w} \wr S^{1}) \times_{\Sigma_{w} \wr S^{1}} (LN)^{w}.
\end{equation}
There is a map from here to the trivial fibration
\[
\mathscr{S}(\bar{F};X)_{\ast} \to \mathscr{S}(\bar{F};X)_{\ast} \times E(\Sigma_{w} \wr S^{1}) \times_{\Sigma_{w} \wr S^{1}} (LN)^{w} \to E(\Sigma_{w} \wr S^{1}) \times_{\Sigma_{w} \wr S^{1}} (LN)^{w}.
\]
On the base space this map is the identity. On the total space the map into the first factor is induced by the inclusion $F \to \bar{F}$; we are able to extend diffeomorphisms over the glued-in discs, at least up to homotopy and we are able to extend maps $(F, \partial_{f}F) \to (X,N)$ over the glued-in discs because any such map must be constant on each window (using the fact that $N$ is discrete). The map into the second factor is given by restriction to windows, as in the previous fibration. On the fibers, the map is given by attaching discs to the extra boundary components and extending diffeomorphisms by the identity and maps to $X$ by $c_{\ast}$, the constant map to the basepoint.

By Theorem \ref{desaru} the map on the fibers is an isomorphism on $H_{s}$ when the genus of $F$ is at least $2s+4$. Thus by the Zeeman comparison theorem we conclude that the map of total spaces is an isomorphism on homology in the same range. Taking limits we get a homology isomorphism
\[
\mathscr{S}_{\infty}^{\infty}(F;X,N)_{\ast} \to \mathscr{S}_{\infty}(\bar{F};X)_{\ast} \times \lim_{w \to \infty}E(\Sigma_{w} \wr S^{1}) \times_{\Sigma_{w} \wr S^{1}} (LN)^{w}.
\]
To conclude, note that since $N$ is discrete, 
\begin{eqnarray*}
E(\Sigma_{w} \wr S^{1}) \times_{\Sigma_{w} \wr S^{1}} (LN)^{w} & \simeq & E\Sigma_{w} \times_{\Sigma_{w}} (ES^1 \times_{S^1} LN)^{w} \\
& \simeq & E\Sigma_{w} \times_{\Sigma_{w}} (BS^1 \times N)^{w}
\end{eqnarray*}
and recall that by the group completion theorem and the Barratt-Priddy-Quillen theorem, there is a homology isomorphism
\[
\mathbb{Z} \times \lim_{w \to \infty}E\Sigma_{w} \times_{\Sigma_{w}} (BS^1 \times N)^{w} \to Q((BS^1 \times N)_{+}).
\]
\end{proof}

The following theorem is relevant.
\begin{theorem}[\cite{Ralph-Ib}, Theorem 1]
\label{lalala}
Suppose that $X$ is simply-connected. For any ordinary, connected cobordism $\bar{F}$ there is a homology isomorphism
\[
h:\mathbb{Z} \times \mathscr{S}_{\infty}(\bar{F};X)_{\ast} \simeq \Omega^{\infty}(MTSO(2) \wedge X_{+}).
\]
\end{theorem}

Note that in the case of general $N$, we still have the fibration (\ref{fibration}) but we don't have a map to the trivial fibration. Perhaps this fibration could be useful in studying $\mathscr{S}(F;X,N)_{\ast}$ for arbitrary $N$.

\subsection{Maps to infinite loop spaces}
\label{infloopmaps}

Let $\C^{b,c}$ denote the full subcategory of $\tCb$ whose objects are
the closed $1$-manifolds. Theorem \ref{thailand} states that
when $X$ is simply-connected and $N$ is discrete, the inclusion
functor induces a weak homotopy equivalence $\Omega B \C^{b,c} \to \Omega B
\tCb$. We now give two maps from $\Omega B \C^{b,c}$ into infinite
loop spaces.

\subsubsection{The map $a: \Omega B \C^{b,c} \to \Omega^{\infty}(MTSO(2) \wedge X_{+})$}

There is a functor $\alpha:\C^{b,c} \to \mathcal{C}^{+}_{2}(X)$ where
$\mathcal{C}^{+}_{2}(X)$ is the ordinary $2$-dimensional cobordism category
with background space $X$, studied in \cite{4author}. These two
categories have a common object set and $\alpha$ is the identity on
the objects.

Given a morphism $[h, \phi]$ in $\C^{b,c}$ with
underlying cobordism $F$, define $\alpha([h, \phi])=[\bar{h},
\bar{\phi}]$ with underlying cobordism $\bar{F}$ where $\bar{F}$ is obtained from $F$ by gluing discs over the windows and $\bar{h}$, $\bar{\phi}$ are obtained from $h$, $\phi$ by extension over the discs. Since the space $\mathrm{Emb}(\bar{F}, \mathbb{R}^{\infty}\times [0,1])$ is contractible, we can make a continuous choice of extension $\bar{h}$. Also, because $N$ is discrete, the map $\phi:(F,\partial_{f}F) \to (X,N)$ must be constant on each window so there is a canonical way to define the extension $\bar{\phi}$.

After applying $\Omega B$, the functor $\alpha$ yields the map $a$
since $\Omega B \mathcal{C}^{+}_{2}(X)$ is homotopy equivalent to
$\Omega^{\infty}(MTSO(2) \wedge X_{+})$ by
\cite{4author}.

\subsubsection{The map $b:\Omega B \C^{b,c} \to Q((BS^{1} \times
N)_{+})$}

By the Barratt-Priddy-Quillen theorem, the target space is homotopy
equivalent to
\begin{eqnarray*}
\Omega B \coprod_{n \geq 0} E\Sigma_{n} \times_{\Sigma_{n}} (BS^{1}
\times N)^{n}.
\end{eqnarray*}
But the $n$th space here is homotopy equivalent to
\[
E\mathrm{Diff}^{+}(\amalg_{n} S^{1}) \times_{\mathrm{Diff}^{+}(\amalg_{n}S^{1})}
\mathrm{Map}(\amalg_{n}S^{1}, N)
\]
using the facts that $N$ is discrete and $\mathrm{Diff}^{+}(\amalg_{n}S^{1})
\simeq \Sigma_{n} \wr S^{1}$.

Thus to define the map $b: \Omega B \C^{b,c} \to Q((BS^{1} \times
N)_{+})$ it suffices to give a functor
\[
\beta: \C^{b,c} \to \coprod_{n \geq 0} E\mathrm{Diff}^{+}(\amalg_{n} S^{1})
\times_{\mathrm{Diff}^{+}(\amalg_{n}S^{1})} \mathrm{Map}(\amalg_{n}S^{1}, N).
\]

Such a functor is given by sending every object in $\C^{b,c}$ to the
unique object in the monoid and is given on the morphisms by
\[
\beta([h, \phi]) = [h \vert_{wind}, \phi \vert_{wind}]
\]
where $\vert_{wind}$ means restriction to the windows of the
underlying open-closed cobordism.

\subsection{Proof of the main theorem}
\label{pfmainthm}

\begin{theorem}
\label{bigtheorem}
When $X$ is simply-connected and $N$ is discrete there is a weak homotopy equivalence
\[
\Omega B \tCb \simeq \Omega^{\infty}(MTSO(2) \wedge X_{+}) \times Q((BS^{1} \times N)_{+}).
\]
\end{theorem}

\begin{proof}
Let $F$ be any connected open-closed cobordism with $\partial_{out}F=S^{1}$ and let $\bar{F}$ denote the ordinary cobordism obtained by gluing discs over the windows of $F$. The maps $a,b,f',g,h$ from Sections \ref{infloopmaps}, \ref{gpcomp} and \ref{splittingsection} fit into the following commutative diagram
\begin{eqnarray*}
\begin{diagram}
\node{\mathbb{Z} \times \mathbb{Z} \times \mathscr{S}^{\infty}_{\infty}(F;X,N)_{\ast}} \arrow{s,l}{f'} \arrow{e,t}{g} \node{\mathbb{Z} \times \mathscr{S}_{\infty}(\bar{F};X)_{\ast} \times Q((BS^{1} \times N)_{+})} \arrow{s,r}{h \times id} \\
\node{\Omega B \C^{b,c}} \arrow{e,b}{a \times b} \node{\Omega^{\infty}(MTSO(2) \wedge X_{+}) \times Q((BS^{1} \times N)_{+})}
\end{diagram}
\end{eqnarray*}
Theorems \ref{malaysia}, \ref{splitting} and \ref{lalala} tell us that the maps $f'$, $g$ and $h$ are homology isomorphisms. Thus $a \times b$ is a homology isomorphism and since the target and source are loop spaces, it is a weak homotopy equivalence. Together with Theorem \ref{thailand}, this proves the theorem.
\end{proof}

Since $B\tCb$ is an infinite loop space, its path-components are all homotopy equivalent. Theorem \ref{bigtheorem} identifies $\Omega B \tCb$ and by delooping, we know the homotopy type of any particular path-component of $B\tCb$. Since we also know $\pi_{0} B\tCb$, we have completely determined $B\tCb$.

We end by noting that as a special case of Theorem \ref{bigtheorem} we have reproved the main result of \cite{BCR}. In that paper the authors studied a $2$-category $\mathscr{S}^{\mathrm{oc}}_{B}$ in which the objects are oriented $1$-manifolds with their boundary points labeled by elements of $B$, the morphisms are oriented open-closed cobordisms with each free boundary component labeled by an element of $B$ and the $2$-morphisms are diffeomorphisms of cobordisms which respect the labels. In the case that $X$ is contractible and $N$ is discrete, $\tCb$ is an embedded surface version of $\mathscr{S}^{\mathrm{oc}}_{B}$ and Theorem \ref{bigtheorem} agrees with the calculation of $\Omega B \mathscr{S}^{\mathrm{oc}}_{B}$ in \cite{BCR}.

\section{Appendix: Homology stability proofs}

In this section we will prove Theorems \ref{homstab1}, \ref{homstab2} and \ref{homstab3}. We will make use of the following generalized form of the Serre spectral sequence.
\begin{theorem}[Special case of the main theorem in \cite{MS}]
Let $F \to E \to B$ be a fibration and $M$ a $\pi_{1}(E)$-module. There is an action of $\pi_{1}(B)$ on $H_{\ast}(F,M)$ and a Serre spectral sequence
\[
E^{2}_{p,q}=H_{p}(B,H_{q}(F,M)) \Rrightarrow H_{p+q}(E,M).
\]
\end{theorem}

\vsp
The notation continues as in section \ref{homstab}. We use $F^{k}_{g,r}$ to denote the surface $F_{g,r}$ with $k$ marked points. 
\begin{lemma}
\label{Noah}
Suppose $X$ is simply-connected and $N$ is discrete. Let $F$ be a connected open-closed cobordism of genus $g$, with $w$ windows and a total of $n+w$ boundary components. There is an isomorphism of $\Gamma_{g,n+w}$-modules
\[
H_{r}(\mathrm{Map}(F, \partial_{f}F; X,N)_{\ast}) \cong \bigoplus_{N^{w}} H_{r}(\mathrm{Map}(F_{g,n+w},\partial F_{g,n+w}; X,\ast))
\]
where the $\Gamma_{g,n+w}$-action on the left-hand side is via the natural inclusion $\Gamma_{g,n+w} \to \Gamma_{\mathrm{oc}}(F)$.
\end{lemma}

\begin{proof}
Because $N$ is discrete, any $f \in \mathrm{Map}(F, \partial_{f}F; X,N)_{\ast}$ must be constant on each window and map the free boundary intervals to the basepoint so
\[
\mathrm{Map}(F, \partial_{f}F; X,N)_{\ast} = \mathrm{Map}(F^{w}_{g,n}, pts; X,N)_{\ast}
\]
where the subscript $\ast$ on the right-hand side indicates that $\partial F_{g,n}^{w}$ must map to the basepoint. If $f, f'$ lie in the same path-component of the mapping space $\mathrm{Map}(F^{w}_{g,n}, pts; X,N)_{\ast}$ then $f(p_{i})=f'(p_{i})$ for all $i$ where $p_{1}, \ldots, p_{w}$ are the marked points of $F^{w}_{g,n}$. Thus
\[
\mathrm{Map}(F^{w}_{g,n}, pts; X,N)_{\ast} = \coprod_{\underline{n} \in N^{w}} \mathrm{Map}^{\underline{n}}(F^{w}_{g,n}, pts; X,N)_{\ast}
\]
where the superscript $\underline{n}$ means that we only allow those $f$ that satisfy $f(p_{1}, \ldots, p_{w})=\underline{n} \in N^{w}$.
Since $X$ is path-connected, for all $\underline{n} \in N^{w}$,
\begin{eqnarray*}
\mathrm{Map}^{\underline{n}}(F^{w}_{g,n}, pts; X,N)_{\ast} & \simeq & \mathrm{Map}^{\underline{\ast}}(F^{w}_{g,n}, pts; X,N)_{\ast} \\
& = & \mathrm{Map}(F_{g,n+w},\partial F_{g,n+w}; X,\ast).
\end{eqnarray*}
The identifications given are all $\Gamma_{g,n+w}$-equivariant.
\end{proof}

\vsp
\begin{proof}[Proof of Theorem \ref{homstab1}.]
Suppose that $F$ has genus $g$, $w$ windows and a total of $n+w$ boundary components. By considering the effect of a diffeomorphism on the windows of $F$ we obtain a map $\mathrm{Diff}_{\mathrm{oc}}^{+}(F) \to \Sigma_{w} \wr \mathrm{Diff}^{+}(S^{1}) \simeq \Sigma_{w} \wr S^{1}$. The kernel is the group of orientation-preserving diffeomorphisms that fix $\partial_{in}F \cup \partial_{out}F$, and the windows, pointwise, call it $\mathrm{Diff}^{\wedge}(F)$. Taking classifying spaces we obtain a fibration
\[
B\mathrm{Diff}^{\wedge}(F) \to B\mathrm{Diff}_{\mathrm{oc}}^{+}(F) \to B(\Sigma_{w} \wr S^{1}).
\]
Using the fact that the diffeomorphism groups have contractible components \cite{EE,ES}, we have a homotopy fibration
\[
B\Gamma_{g,n+w} \to B\Gamma_{\mathrm{oc}}(F) \to B(\Sigma_{w} \wr S^{1}).
\]
The inclusion $F \to F^{\#}$ induces a map from this fibration to
\[
B\Gamma_{g^{\#},n^{\#}+w} \to B\Gamma_{\mathrm{oc}}(F^{\#}) \to B(\Sigma_{w} \wr S^{1}).
\]

Consider the Serre spectral sequences for these fibrations, taken with twisted coefficients in $H_{r}(
\mathrm{Map}(F, \partial_{f}F; X,N)_{\ast})$, $H_{r}(
\mathrm{Map}(F^{\#}, \partial_{f}F^{\#}; X,N)_{\ast})$ respectively. There is an induced map of spectral sequences and on the $E^{2}$-page it is of the form
\begin{eqnarray*}
&&E^{2}_{p,q} = H_{p}(B(\Sigma_{w} \wr S^{1}); H_{q}(\Gamma_{g,n+w}, H_{r}(\mathrm{Map}(F, \partial_{f}F; X,N)_{\ast}))) \\
&& \to \tilde{E}^{2}_{p,q} = H_{p}(B(\Sigma_{w} \wr S^{1}); H_{q}(\Gamma_{g^{\#},n^{\#}+w}, H_{r}(\mathrm{Map}(F^{\#}, \partial_{f}F^{\#}; X,N)_{\ast}))).
\end{eqnarray*}
By Lemma \ref{Noah}, this is isomorphic to the map
\begin{eqnarray*}
&& H_{p}(B(\Sigma_{w} \wr S^{1}); \bigoplus_{N^{w}}H_{q}(\Gamma_{g,n+w}, H_{r}(\mathrm{Map}(F_{g,n+w},\partial F_{g,n+w}; X,\ast)))) \\
&& \to H_{p}(B(\Sigma_{w} \wr S^{1}); \bigoplus_{N^{w}}H_{q}(\Gamma_{g^{\#},n^{\#}+w}, H_{r}(\mathrm{Map}(F_{g^{\#},n^{\#}+w},\partial F_{g^{\#},n^{\#}+w}; X,\ast))))
\end{eqnarray*}
which, by Theorem \ref{india}, is an isomorphism for $g \geq 2q+r+2$. The map on the abutments is precisely the map in the statement of the theorem so applying the Zeeman comparison theorem yields the result.
\end{proof}

\vsp
\begin{proof}[Proof of Theorem \ref{homstab2}.]
Let $g,n,w$ be as above. As in the proof of Theorem \ref{homstab1} we have a map of fibrations
\begin{eqnarray}
\begin{diagram}
\node{B\Gamma_{g,n+w}} \arrow{e} \arrow{s} \node{B\Gamma_{\mathrm{oc}}(F)} \arrow{e} \arrow{s} \node{B(\Sigma_{w} \wr S^{1})}  \arrow{s} \\
\node{B\Gamma_{g,n+w+1}} \arrow{e} \node{B\Gamma_{\mathrm{oc}}(F^{\#})} \arrow{e} \node{B(\Sigma_{w+1} \wr S^{1})} \label{diagram}
\end{diagram}
\end{eqnarray}
Again we look at the twisted Serre spectral sequences associated to these fibrations and the induced map between them. On the $E^{2}$-page this now takes the form
\begin{eqnarray*}
&& E^{2}_{p,q} = H_{p}(B(\Sigma_{w} \wr S^{1}); H_{q}(\Gamma_{g,n+w}, H_{r}(\mathrm{Map}(F, \partial_{f}F; X,N)_{\ast})))
\\ && \to \tilde{E}^{2}_{p,q} = H_{p}(B(\Sigma_{w+1} \wr S^{1}); H_{q}(\Gamma_{g,n+w+1}, H_{r}(\mathrm{Map}(F^{\#}, \partial_{f}F^{\#}; X,N)_{\ast}))). 
\end{eqnarray*}
Since $N$ is a point, the coefficient groups in these homology groups are the same as
\begin{eqnarray*}
& H_{q}(\Gamma_{g,n+w}, H_{r}(\mathrm{Map}(F_{g,n+w},\partial F_{g,n+w}; X,\ast))),\\ 
& H_{q}(\Gamma_{g,n+w+1}, H_{r}(\mathrm{Map}(F_{g,n+w},\partial F_{g,n+w}; X,\ast)))
\end{eqnarray*}
so by Theorem \ref{india} the map of coefficients is an isomorphism when $g \geq 2q+r+2$. By the next lemma, the $\Sigma_{w}$-action on these coefficients is trivial in this range. Thus, in the stable range, the map on the $E^{2}$-page is of the form
\[
H_{p}(B(\Sigma_{w} \wr S^{1}), G) \to H_{p}(B(\Sigma_{w+1} \wr S^{1}), G)
\]
for a constant, untwisted coefficient group $G$. Since $B(\Sigma_{w} \wr S^{1}) \simeq E \Sigma_{w} \times_{\Sigma_{w}} (BS^{1})^{w}$ this is an isomorphism for $w \geq 2p+6$, see \cite{nonoble} Theorem 5. An application of the Zeeman comparison theorem completes the proof.
\end{proof}

\begin{lemma}
\label{Catatonia}
For $g \geq 2q+r+2$, the $\Sigma_{w}$-action on
\[
H_{q}(\Gamma_{g,n+w}, H_{r}(\mathrm{Map}(F_{g,n+w},\partial F_{g,n+w}; X,\ast)))
\]
is trivial.
\end{lemma}

\begin{proof}
The $\Sigma_{w}$-action is induced by the first fibration in (\ref{diagram}). Thus it is given by lifting a permutation $\sigma \in \Sigma_{w}$ to a mapping class $\phi_{\sigma} \in \Gamma_{\mathrm{oc}}(F)$ and then acting by conjugation on $\Gamma_{g,n+w}$ and by left-multiplication on $H_{r}(\mathrm{Map}(F_{g,n+w},\partial F_{g,n+w}; X,\ast))$.

Think of $F$ as $F_{1} \# F_{2}$ where $F_{1}$ has no windows and $F_{2}$ is a cylinder with $w$ windows. Similarly, think of $F_{g,n+w}$ as $F_{g,n} \# F_{0,w+2}$. The inclusion $\Gamma_{g,n+w} \to \Gamma_{\mathrm{oc}}(F)$ restricts to an inclusion $\Gamma_{g,n} \to \Gamma_{\mathrm{oc}}(F_{1})$. By Theorem \ref{india}, there is an isomorphism
\begin{eqnarray*}
&& H_{q}(\Gamma_{g,n}, H_{r}(\mathrm{Map}(F_{g,n},\partial F_{g,n}; X,\ast))) \\
&& \hspace{8mm} \to H_{q}(\Gamma_{g,n+w}, H_{r}(\mathrm{Map}(F_{g,n+w},\partial F_{g,n+w}; X,\ast)))
\end{eqnarray*}
so we can represent an arbitrary element of the latter group by $z \otimes v$ from the bar resolution
\[
\mathrm{Bar}(\Gamma_{g,n}) \otimes_{\Gamma_{g,n}} H_{r}(\mathrm{Map}(F_{g,n},\partial F_{g,n}; X,\ast)).
\]
To calculate $\sigma \cdot (z \otimes v)$ choose $\phi_{\sigma} \in \Gamma_{\mathrm{oc}}(F_{2}) \subseteq \Gamma_{\mathrm{oc}}(F)$. Since $\phi_{\sigma}$ and the diffeomorphisms appearing in $z$ have disjoint support, they commute. For similar reasons $\phi_{\sigma} \cdot v = v$ so $\sigma \cdot (z \otimes v)=z \otimes v$.
\end{proof}

\vsp
\begin{proof}[Proof of Theorem \ref{homstab3}.]
Suppose for simplicity that $F$ itself has no windows so that $F^{w}$ has $w$ windows and $(F^{\#})^{w}$ has $w+1$ windows (see the introduction for the definition of $F^{w}$ and the natural inclusion $F^{w} \to F^{w+1}$). Suppose further that $F$ has genus $g$ and $n$ boundary components; thus $F^{w}$ has genus $g+w$ and a total of $n+w$ boundary components. Let $\mu_{0}, \mu_{1}: \Sigma_{w} \wr S^{1} \to \Sigma_{w+1} \wr S^{1}$ be defined by $\mu_{0}(\sigma, \underline{z}) = (\iota \times \sigma, 1 \times \underline{z})$ and $\mu_{1}(\sigma, \underline{z}) = (\sigma \times \iota, \underline {z} \times 1)$ where $\iota \in \Sigma_{1}$ is the identity permutation.

Taking the limit of diagram (\ref{diagram}) over $F^{w}, (F^{\#})^{w}$ we obtain the following diagram of fibrations
\begin{eqnarray*}
\begin{diagram}
\node{B\Gamma_{\infty,\infty}} \arrow{e} \arrow{s} \node{B\Gamma^{\infty}_{\infty}(F)} \arrow{e} \arrow{s} \node{\mathrm{hocolim}_{w \to \infty} \hspace{1ex} B(\Sigma_{w} \wr S^{1})}  \arrow{s} \\
\node{B\Gamma_{\infty,\infty+1}} \arrow{e} \node{B\Gamma^{\infty}_{\infty}(F^{\#})} \arrow{e} \node{\mathrm{hocolim}_{w \to \infty} \hspace{1ex}B(\Sigma_{w+1} \wr S^{1})}
\end{diagram}
\end{eqnarray*}

Consider the Serre spectral sequences associated to these two fibrations, with twisted coefficients taken in
\[
H_{r}(\mathrm{Map}_{\infty}^{\infty}(F, \partial_{f}F; X,N)_{\ast}) \hspace{3mm} \text{and} \hspace{3mm} H_{r}(\mathrm{Map}_{\infty}^{\infty}(F^{\#}, \partial_{f}F^{\#}; X,N)_{\ast})
\]
respectively. There is an induced map between these spectral sequences and the map on the abutments is precisely the one in the theorem. Thus, by the Zeeman comparison theorem, it suffices to show that the map on the $E^{2}$-page is an isomorphism. The rest of the proof is devoted to that task.

The $E^{2}_{p,q}$-terms are
\begin{eqnarray*}
& H_{p}(\mathrm{hocolim} \hspace{1ex} B(\Sigma_{w} \wr S^{1}); H_{q}(B\Gamma_{\infty,\infty}; H_{r}(\mathrm{Map}_{\infty}^{\infty}(F, \partial_{f}F; X,N)_{\ast}))) \\
& H_{p}(\mathrm{hocolim} \hspace{1ex} B(\Sigma_{w+1} \wr S^{1}); H_{q}(B\Gamma_{\infty,\infty+1}); H_{r}(\mathrm{Map}_{\infty}^{\infty}(F^{\#}, \partial_{f}F^{\#}; X,N)_{\ast}))).
\end{eqnarray*}

Since limits and group homology commute these are isomorphic to the limits, as $w \to \infty$, of
\begin{eqnarray*}
\label{ocs}
& H_{p}(B(\Sigma_{w}\wr S^{1}); H_{q}(\Gamma_{g+w,n+w}; H_{r}(\mathrm{Map}(F^{w}, \partial_{f}F^{w}; X,N)_{\ast}))), \\
& H_{p}(B(\Sigma_{w+1}\wr S^{1}); H_{q}(\Gamma_{g+w,n+w+1};H_{r}(\mathrm{Map}((F^{\#})^{w}, \partial_{f}(F^{\#})^{w}; X,N)_{\ast}))).
\end{eqnarray*}
It will be important for us to keep track of the maps used to form limits. The maps used here are induced by $B\mu_{1}$ on the spaces and by $F^{w} \hookrightarrow F^{w+1}$ on the coefficients. 

The map between the $E^{2}$-terms can be realized by finite-level maps
\begin{eqnarray}
\label{beatles}
&& H_{p}(B(\Sigma_{w}\wr S^{1}); H_{q}(\Gamma_{g+w,n+w}; H_{r}(\mathrm{Map}(F^{w}, \partial_{f}F^{w}; X,N)_{\ast}))) \\
\nonumber
&& \to H_{p}(B(\Sigma_{w+1}\wr S^{1}); H_{q}(\Gamma_{g+w,n+w+1}; H_{r}(\mathrm{Map}((F^{\#})^{w}, \partial_{f}(F^{\#})^{w}; X,N)_{\ast}))).
\end{eqnarray}
These are induced on the spaces by $B\mu_{0}$ and on the coefficients by $F^{w} \hookrightarrow (F^{\#})^{w}$.

By Lemma \ref{Noah}, the group in (\ref{beatles}) is isomorphic to
\[
H_{p}(B(\Sigma_{w} \wr S^{1}); \bigoplus_{\underline{n} \in N^{w}} H_{q}(\Gamma_{g+w,n+w}; H_{r}(\mathrm{Map}(F_{g+w,n+w},\partial F_{g+w,n+w}; X,\ast)))).
\]
By Lemma \ref{Catatonia} we know that the $\Sigma_{w}$-action on each of these summands is trivial provided the genus is sufficiently large. Furthermore, by Theorem \ref{india} we know that $H_{q}(\Gamma_{g+w,n+w}; H_{r}(\mathrm{Map}(F_{g+w,n+w},\partial F_{g+w,n+w}; X,\ast)))$ is independent of $g,n$ and $w$ for large genus. We conclude that for sufficiently large $w$, the group in (\ref{beatles}) is of the form
\[
H_{p}(B(\Sigma_{w} \wr S^{1}); \bigoplus_{\underline{n} \in N^{w}} A)
\]
where $A$ is an abelian group with trivial $\Sigma_{w}$-action and the action of $\Sigma_{w}$ on the coefficients is induced by the permutation action on $N^{w}$.

Let $\nu_{0}: \bigoplus_{\underline{n} \in N^{w}} A \to \bigoplus_{\underline{n} \in N^{w+1}} A$ be the map induced by $\tilde{\nu}_{0}: N^{w} \to N^{w+1}$ with $\tilde{\nu}_{0}(\underline{n}) = \ast \times \underline{n}$. Similarly we have $\nu_{1}: \bigoplus_{\underline{n} \in N^{w}} A \to \bigoplus_{\underline{n} \in N^{w+1}} A$ induced by $\tilde{\nu}_{1}$ with $\tilde{\nu}_{1}(\underline{n})= \underline{n} \times \ast$.

Now, 
\begin{eqnarray*}
\label{cast}
& E^{2}_{p,q} \cong \lim_{w \to \infty} H_{p}(B(\Sigma_{w} \wr S^{1}); \bigoplus_{\underline{n} \in N^{w}}A) \\
\label{cast2}
& \tilde{E}^{2}_{p,q} \cong \lim_{w \to \infty} H_{p}(B(\Sigma_{w+1} \wr S^{1}); \bigoplus_{\underline{n} \in N^{w+1}}A).
\end{eqnarray*}
and the maps used to form these limits are $H_{p}(B\mu_{1}, \nu_{1})$. Under these identifications, the map between the $E^{2}$-pages is induced by the finite-level maps $H_{p}(B \mu_{0}, \nu_{0})$.

Let $\tau \in \Sigma_{w+1}$ be the block permutation which sends $(1,2,\ldots,w+1)$ to $(2, \ldots, w+1,1)$. Let $c_{(\tau,1)}:\Sigma_{w+1} \wr S^{1} \to \Sigma_{w+1} \wr S^{1}$ denote conjugation by $(\tau,1)$ and $\rho_{\tau}: \bigoplus_{\underline{n} \in N^{w+1}} A \to \bigoplus_{\underline{n} \in N^{w+1}} A$ be given by the action of $\tau$. Then $c_{(\tau,1)}\circ \mu_{0} = \mu_{1}$ and $\rho_{\tau} \circ \nu_{0}=\nu_{1}$ so
\[
H_{p}(Bc_{(\tau,1)},\rho_{\tau}) \circ H_{p}(B \mu_{0}, \nu_{0}) = H_{p}(B\mu_{1}, \nu_{1}).
\]
It is a fundamental property of group cohomology that $H_{p}(Bc_{(\tau,1)},\rho_{\tau})$ is the trivial map. Thus $H_{p}(B \mu_{0}, \nu_{0})$ is the same as $H_{p}(B\mu_{1}, \nu_{1})$ and hence it is trivial in the limit since $H_{p}(B\mu_{1}, \nu_{1})$ is used to form that limit. We conclude that the map on the $E^{2}$-page is an isomorphism.
\end{proof}



\bibliographystyle{plain}

\bibliography{thesis}

\end{document}